\documentclass[10pt,leqno]{amsart}

\usepackage{amssymb}
\usepackage{graphicx}
\usepackage{enumitem}
\usepackage{subcaption}
\usepackage{booktabs}
\usepackage{algorithm}
\usepackage{algorithmic}
\usepackage[numbers,sort&compress]{natbib}
\usepackage[hidelinks]{hyperref}
\usepackage{orcidlink}
\usepackage[capitalize,nameinlink]{cleveref}

\graphicspath{{./}}
\DeclareGraphicsExtensions{.pdf,.png,.jpg}

\numberwithin{equation}{section}
\numberwithin{algorithm}{section}

\newtheorem{theorem}{Theorem}[section]
\newtheorem{proposition}[theorem]{Proposition}
\newtheorem{lemma}[theorem]{Lemma}
\newtheorem{corollary}[theorem]{Corollary}

\theoremstyle{remark}
\newtheorem{remark}[theorem]{Remark}

\newtheorem{fact}[theorem]{Fact}
\newtheorem{assumption}[theorem]{Assumption}
\crefname{hypothesis}{Hypothesis}{Hypotheses}
\crefname{fact}{Fact}{Facts}

\newcommand{\Hcal}{\mathcal{H}}
\newcommand{\Kcal}{\mathcal{K}}
\newcommand{\R}{\mathbb{R}}
\newcommand{\Id}{\operatorname{Id}}
\newcommand{\Fix}{\operatorname{fix}}
\newcommand{\dom}{\operatorname{dom}}

\newcommand{\subd}{\partial}

\DeclareMathOperator{\prox}{prox}
\DeclareMathOperator{\argmin}{arg\,min}
\DeclareMathOperator{\dist}{dist}

\DeclareMathOperator{\diag}{diag}

\newcommand{\norm}[1]{\left\lVert #1 \right\rVert}
\newcommand{\ip}[2]{\left\langle #1,#2 \right\rangle}

\newcommand{\grad}{\nabla}

\title[Curvature-Recycling Douglas--Rachford Splitting]{Curvature Recycling Douglas-Rachford Splitting: Transported Quasi-Newton Models for Expensive Smooth Proximal Subproblems}
\author[M. Maharramov]{Musa Maharramov\,\orcidlink{0000-0002-4323-0987}}
\thanks{Research profiles: \href{https://orcid.org/0000-0002-4323-0987}{ORCID iD 0000-0002-4323-0987}; \href{https://scholar.google.com/citations?user=aIpqtCsAAAAJ\&hl=en}{Google Scholar}; \href{https://www.researchgate.net/profile/Musa_Maharramov}{ResearchGate}}
\address{Industrial Mathematics LLC, Spring, Texas, USA}
\email{musa@industrialmathematics.com}
\urladdr{https://www.industrialmathematics.com/}
\date{}
\hypersetup{
  pdftitle={Curvature Recycling Douglas-Rachford Splitting: Transported Quasi-Newton Models for Expensive Smooth Proximal Subproblems},
  pdfauthor={Musa Maharramov},
  pdfsubject={Curvature-recycling Douglas-Rachford splitting},
  pdfkeywords={Douglas-Rachford splitting, Peaceman-Rachford splitting, ADMM, proximal methods, quasi-Newton methods, convex optimization}
}

\begin{document}

\begin{abstract}
We consider
\[
    \min_x f(x)+g(x),
\]
where $f$ is smooth but its value and gradient are expensive, while $g$ is nonsmooth and has a cheap proximal map.  Douglas--Rachford splitting (DRS) then requires a sequence of smooth proximal solves.  These solves have different centers but share the same nonlinear curvature.

We introduce curvature-recycling DRS (CR-DRS).  The method transports an old exact residual to the new proximal center, reuses a quasi-Newton model, and moves the old proximal state before the first new gradient is evaluated.  We give paired direct and inverse BFGS and safeguarded symmetric-rank-one realizations.  A certified variant accepts a transported quasi-Newton step only when it passes descent and residual tests; otherwise it uses a fixed number of safe gradient steps.  Under strong convexity and smoothness, a small-gain condition gives global linear convergence with a constant number of new gradients per outer iteration.  A transported Dennis--Mor\'e condition gives local superlinear reduction of the new proximal error.  We show that this condition is not automatic, derive full-model and active-subspace sufficient conditions, and relate the active subspace to the local DRS dynamics induced by the nonsmooth proximal map.

Experiments on $\ell_1$- and total-variation-regularized problems show large reductions in expensive-gradient calls relative to restarted and curvature-only quasi-Newton proximal solves.  For directly proximal $\ell_1$ models, we also compare with FISTA.
\end{abstract}

\maketitle

\section{Introduction}
\label{sec:introduction}

We study
\begin{equation}
\label{eq:main-problem}
    \min_{x\in\Hcal} F(x):=f(x)+g(x),
\end{equation}
where $\Hcal$ is a finite-dimensional real Hilbert space, $f:\Hcal\to\R$ is convex and twice continuously differentiable, and $g:\Hcal\to\R\cup\{+\infty\}$ is proper, closed, and convex.  The cost model is asymmetric.  One call to $f$ or $\grad f$ may require a simulation or a forward/adjoint solve.  Vector operations, limited-memory matrix products, and $\prox_{\gamma g}$ are assumed much cheaper.

For $\gamma>0$ and relaxation $\alpha>0$, one DRS step is
\begin{align}
    p_k&=P_{\gamma f}z_k,
    &P_{\gamma f}z&:=\argmin_x\left\{\gamma f(x)+\frac12\norm{x-z}^2\right\},
    \label{eq:intro-proxf}\\
    y_k^g&=P_{\gamma g}(2p_k-z_k),
    \label{eq:intro-proxg}\\
    z_{k+1}&=z_k+2\alpha(y_k^g-p_k).
    \label{eq:intro-drs-update}
\end{align}
The usual DRS convention is $\alpha=1/2$; $\alpha=1$ gives the Peaceman--Rachford map.  The exact theory is recalled in Appendices~\ref{app:exact-drs} and~\ref{app:linear-drs}.

The difficult operation is often \eqref{eq:intro-proxf}.  At outer iteration $k$, it is the minimization of
\begin{equation}
\label{eq:intro-prox-family}
    \phi_k(x):=\gamma f(x)+\frac12\norm{x-z_k}^2.
\end{equation}
The proximal point is equivalently characterized by setting the residual
\begin{equation}
\label{eq:intro-prox-family-res}
    r_k(x):=\grad\phi_k(x)=\gamma\grad f(x)+x-z_k.
\end{equation}

The center $z_k$ changes.  The curvature does not:
\begin{equation}
\label{eq:intro-curvature-invariance}
    \grad^2\phi_k(x)=\gamma\grad^2 f(x)+I.
\end{equation}
This is the main structural fact.  It gives two exact reuse rules.  First, an old residual can be shifted to a new center without another gradient.  Second, secant pairs from earlier proximal solves remain valid because the center cancels from residual differences.

The symbols and fixed parameters in the algorithm are collected in Section~\ref{sec:notation-parameters}. 

CR-DRS uses these rules as follows.  

\begin{algorithm}[t]
\caption{CR-DRS at a high level}
\label{alg:high-level-crdrs}
\begin{algorithmic}[1]
\REQUIRE Outer state $z_k$; exact-gradient anchor $a_{k-1}$; memory $\mathcal M_k$ and its paired direct/inverse Hessian model $(B_k,H_k)$; parameters from Section~\ref{sec:notation-parameters}.
\STATE {\bf Predictor phase:} Transport the old residual to the current proximal problem:
       $\bar r_k=\gamma\grad f(a_{k-1})+a_{k-1}-z_k$.
\STATE Apply $K$ cheap curvature-correction model steps using $H_k$ for directions, $B_k$ for residual correction, and the no-line-search damping rule \eqref{eq:predictor-eta}; obtain $\widehat x_k$.
\STATE {\bf Corrector phase:} Evaluate one true current gradient $\grad f(\widehat x_k)$.
\STATE Accept $\widehat x_k$ if descent and residual certificates hold; otherwise use a fixed safe fallback block.
\STATE Form an exact secant pair and update $(B_k,H_k)$ by BFGS or a safeguarded symmetric-rank-one formula.
\STATE Apply $P_{\gamma g}$ and update $z_{k+1}$ by \eqref{eq:intro-drs-update}.
\end{algorithmic}
\end{algorithm}

The $K$ model steps do not call $\grad f$.  With a frozen paired model, they collapse to one damped quasi-Newton direction.  Their practical role is to continue a clipped step toward the current model minimizer.  The first new gradient checks the model and supplies a new exact curvature pair.

The paper has four main results.
\begin{enumerate}[label=(\roman*),leftmargin=2.2em]
    \item We give a precise CR-DRS algorithm and two curvature realizations: paired BFGS and safeguarded symmetric rank one (SR1).  Both retain only secants with exact-gradient endpoints.
    \item We give a certified fixed-budget method.  A transported quasi-Newton candidate is accepted only if it passes an Armijo-type descent test and a computable residual-contraction test.  A fixed number of gradient steps is used as a fallback.  The resulting inner call is uniformly contractive.
    \item We state the global theorem first: under the Giselsson--Boyd assumptions and a small-gain condition, the coupled outer DRS error and inner proximal-tracking error converge linearly.  The proof is then split into an inner contraction, an outer perturbation bound, and a two-state recursion.
    \item We adapt the Dennis--Mor\'e argument to the transported error direction.  Exact direct/inverse pairing removes a separate inversion-consistency hypothesis.  We give a DRS-realizable obstruction, formulate full-model and active-subspace sufficient conditions, and show how local regularity of $g$ confines the outer right-hand-side corrections to a fixed DRS-generated subspace.  For quadratic curvature, accepted SR1 secants recover the Hessian on any finitely excited active subspace.
\end{enumerate}

This work is related to three established lines of research.  Exact and inexact DRS, PRS, proximal-point, and ADMM theory provides the outer fixed-point framework and admissible error models \cite{DouglasRachford1956,LionsMercier1979,Rockafellar1976,EcksteinBertsekas1992,BauschkeCombettes2017,AlvesEcksteinGeremiaMelo2020}.  Linear DRS bounds and parameter choices are taken from Giselsson and Boyd \cite{GiselssonBoyd2017}.  Local manifold identification for partly smooth regularizers explains why many DRS/ADMM trajectories eventually evolve on a fixed low-dimensional tangent structure \cite{LiangFadiliPeyreLuke2015}.  Classical quasi-Newton theory provides BFGS, SR1, and Dennis--Mor\'e superlinear analysis for fixed objectives \cite{BroydenDennisMore1973,dennismore1974,DennisMore1977,NocedalWright2006}.  The difference here is the moving proximal target: the objective changes by an affine term at every outer step, and the quasi-Newton model is transported before the first current gradient.

Section~\ref{sec:method} defines CR-DRS and its BFGS and SR1 forms.  Section~\ref{sec:global} proves global convergence for the certified fixed-budget method.  Section~\ref{sec:superlinear} gives the local transported quasi-Newton result.  Section~\ref{sec:dm-coverage} studies when the transported Dennis--Mor\'e condition follows from full or active-subspace curvature learning and connects that subspace to the local DRS dynamics generated by $g$.  Section~\ref{sec:heuristics} records the practical predictor and safeguards.  Section~\ref{sec:numerics} reports the numerical tests.  Appendices~\ref{app:proximal-facts}--\ref{app:linear-drs} provide the proximal and exact-DRS prerequisites.

\section{Curvature-recycling DRS with BFGS and rank-one updates}
\label{sec:method}

\subsection{Notation and algorithmic parameters}
\label{sec:notation-parameters}

The proximal and reflected proximal maps $P_{\gamma h}$ and $R_{\gamma h}$ are defined in \eqref{eq:proximal-definition}--\eqref{eq:reflected-proximal-definition}.  The identity operator is $I=\Id$.  For self-adjoint operators, $A\preceq B$ means $\ip{x}{Ax}\le \ip{x}{Bx}$ for every $x$.  The transpose of a matrix is denoted by ${}^{\top}$, the Hilbert-space adjoint of a bounded linear map by ${}^{*}$, and $\varrho(M)$ denotes the spectral radius of a square matrix $M$.  We use
\[
  \operatorname{clip}_{[a,b]}(t):=\min\{b,\max\{a,t\}\}.
\]

The following notation is used throughout the body of the paper.
\begin{description}[style=nextline,leftmargin=2.8em,labelindent=0em]
\item[Outer and predictor indices.]
$k=0,1,\ldots$ is the DRS outer index.  The index $j=0,\ldots,K-1$ labels the cheap predictor steps.  The predictor depth $K$ is a nonnegative integer.  By convention, $K=0$ means no transported displacement: $\widehat x_k=a_{k-1}$ and $\eta_k^{\rm eff}=0$.

\item[DRS variables.]
$z_k$ is the DRS fixed-point state, $p_k=P_{\gamma f}z_k$ is the exact smooth proximal point, $x_k$ is the accepted approximation of $p_k$, and $y_k^g=P_{\gamma g}(2x_k-z_k)$ is the exact nonsmooth proximal output.  The proximal scale is $\gamma>0$, and $\alpha>0$ is the outer relaxation in $z_{k+1}=z_k+2\alpha(y_k^g-x_k)$.  The admissible ranges of $\alpha$ are stated in Appendices~\ref{app:exact-drs} and~\ref{app:linear-drs}.

\item[Exact anchor and memory.]
$a_{k-1}$ is an old point with a stored exact value $f(a_{k-1})$ and exact gradient $\grad f(a_{k-1})$; after a completed inner call we set $a_k=x_k$.  The initial anchor $a_{-1}$ is supplied by the user.  The memory cap is $m_q\in\{1,2,\ldots\}\cup\{\infty\}$.  When $m_q<\infty$, the ordered memory $\mathcal M_k$ contains at most $m_q$ accepted exact secant pairs $(s_i,q_i)$, where both endpoints have true gradients.  The choice $m_q=\infty$ denotes full-memory BFGS or SR1 and retains all accepted pairs.

\item[Curvature models.]
$H_k\succ0$ is an inverse-Hessian approximation for $\grad^2\phi_k$, and $B_k\succ0$ is its paired direct approximation.  The reference construction starts from $H_{k,0}=\vartheta_k I$ and $B_{k,0}=\vartheta_k^{-1}I$, where $\vartheta_k\in[\underline h,\overline h]$ and $0<\underline h\le\overline h<\infty$, and then replays the pairs in $\mathcal M_k$.  If the memory is nonempty and $(s_\ell,q_\ell)$ is its newest pair, we use
\begin{equation}
\label{eq:initial-qn-scaling}
  \vartheta_k
  :=\operatorname{clip}_{[\underline h,\overline h]}
       \left(\frac{\ip{s_\ell}{q_\ell}}{\norm{q_\ell}^2}\right);
\end{equation}
otherwise $\vartheta_k=\vartheta_0$ for a prescribed $\vartheta_0\in[\underline h,\overline h]$.

\item[Predictor variables and damping.]
$\widetilde x_{k,j}$ and $\widetilde r_{k,j}$ are the pseudo-state and modeled residual, $d_{k,j}$ is the quasi-Newton direction, and $\eta_{k,j}$ is its scalar damping.  The parameter $\bar\eta_k\in(0,1]$ is a cap fixed before the current rollout from stored information only, and $\nu_{\rm qn}\in(0,2)$ is the model-descent fraction in \eqref{eq:predictor-eta}.  The certified reference algorithm uses the constant policy $\bar\eta_k\equiv\bar\eta$ for a prescribed $\bar\eta\in(0,1]$; the past-data adaptive alternative is \eqref{eq:damping-trust-switch}.  The effective damping $\eta_k^{\rm eff}$ is defined in \eqref{eq:frozen-endpoint-new}.

\item[Pair and model safeguards.]
The constants $\epsilon_{\rm curv}\in[0,1)$ and $\epsilon_{\rm sr1}\in(0,1)$ are, respectively, the exact-pair curvature threshold \eqref{eq:pair-curvature-threshold} and the SR1 denominator threshold \eqref{eq:sr1-skip}.  The spectral bounds $\underline h$ and $\overline h$ are also used in \eqref{eq:sr1-spectral-bounds}.

\item[Certified inner call.]
The constant $c_A\in(0,1/2)$ is the one-shot Armijo parameter, $\theta_Q\in(0,1)$ is the requested contraction factor for an accepted quasi-Newton candidate, and $b\ge1$ is the fixed number of safe gradient steps used after rejection.  The strong-convexity and smoothness constants $m$ and $L$, and the fallback step $\tau_G$, are defined in \eqref{eq:m-L-prox} and \eqref{eq:safe-gradient-step}.

\item[Convergence terminology.]
A sequence $u_k$ converges $R$-linearly to $u_\star$ if $\norm{u_k-u_\star}\le Cq^k$ for some $C>0$ and $q\in(0,1)$.  A transported proximal candidate is called $Q$-superlinear relative to the ordinary warm start when the ratio in \eqref{eq:transport-superlinear-focused} tends to zero.  The symbols $O(\cdot)$ and $o(\cdot)$ have their usual deterministic asymptotic meanings.
\end{description}

\subsection{Exact transport and reusable secants}
\label{sec:transport-secant}

Let
\begin{equation}
\label{eq:method-prox-family}
    \phi_k(x)=\gamma f(x)+\frac12\norm{x-z_k}^2,
    \qquad
    r_k(x)=\gamma\grad f(x)+x-z_k,
    \qquad
    p_k=P_{\gamma f}z_k.
\end{equation}
Suppose $a$ is a point at which $\grad f(a)$ has already been evaluated.  Its exact residual for the current proximal problem is
\begin{equation}
\label{eq:exact-residual-transport}
    \bar r_k(a):=\gamma\grad f(a)+a-z_k.
\end{equation}
If $a$ was stored under center $z_j$ together with $r_j(a)$, then
\begin{equation}
\label{eq:residual-shift-simple}
    \bar r_k(a)=r_j(a)+z_j-z_k.
\end{equation}
No new call to $\grad f$ is needed.

For two exact-gradient points $u$ and $v$, set
\begin{equation}
\label{eq:reusable-secant}
    s=v-u,
    \qquad
    y=r_k(v)-r_k(u)
      =\gamma\bigl(\grad f(v)-\grad f(u)\bigr)+s.
\end{equation}
The right-hand side is independent of $z_k$.  Thus the pair $(s,y)$ can be used at every later center.  In the indexed memory notation of Section~\ref{sec:notation-parameters}, this same curvature vector is denoted by $q_i$ in the stored pair $(s_i,q_i)$.

\begin{proposition}[Exact reuse]
\label{prop:exact-reuse}
For fixed $\gamma$, \eqref{eq:residual-shift-simple} and \eqref{eq:reusable-secant} are exact.  If $f$ is convex and $s\ne0$, then
\begin{equation}
\label{eq:automatic-curvature}
    \ip{s}{y}
    =\gamma\ip{s}{\grad f(v)-\grad f(u)}+\norm{s}^2
    \ge \norm{s}^2>0.
\end{equation}
\end{proposition}

\begin{proof}
Equation~\eqref{eq:residual-shift-simple} follows by subtracting the two centers in \eqref{eq:method-prox-family}.  Equation~\eqref{eq:reusable-secant} follows by subtracting two residuals with the same center.  Monotonicity of $\grad f$ gives \eqref{eq:automatic-curvature}.
\end{proof}

\begin{remark}[Changing the proximal parameter]
If $\gamma$ changes, store the smooth pair
$y^f=\grad f(v)-\grad f(u)$ and rebuild the current proximal pair as
$y(\gamma)=\gamma y^f+s$.  The center still does not enter the pair.
\end{remark}

\subsection{The transported predictor}
\label{sec:transported-predictor}

Let $H_k$ be an inverse-Hessian model and let $B_k$ be its paired direct model.  Starting from an exact-gradient anchor $a_{k-1}$, define
\begin{equation}
\label{eq:model-rollout-initial}
    \widetilde x_{k,0}=a_{k-1},
    \qquad
    \widetilde r_{k,0}=\bar r_k(a_{k-1}).
\end{equation}
At pseudo-step $j$, set
\begin{equation}
\label{eq:model-rollout-direction}
    d_{k,j}:=-H_k\widetilde r_{k,j}.
\end{equation}
The local quadratic change predicted by $B_k$ is
\begin{equation}
\label{eq:pseudo-quadratic-model}
    \mathcal Q_{k,j}(s)
    :=\ip{\widetilde r_{k,j}}{s}
      +\frac12\ip{s}{B_ks}.
\end{equation}
We use the following closed-form damping rule:
\begin{equation}
\label{eq:predictor-eta}
\eta_{k,j}:=
\begin{cases}
\displaystyle
\min\left\{\bar\eta_k,
\nu_{\rm qn}
\frac{-\ip{\widetilde r_{k,j}}{d_{k,j}}}
     {\ip{d_{k,j}}{B_kd_{k,j}}}
\right\},
&\begin{array}{l}
 d_{k,j}\ne0,\\[-1mm]
 \ip{\widetilde r_{k,j}}{d_{k,j}}<0,\\[-1mm]
 \ip{d_{k,j}}{B_kd_{k,j}}>0,
\end{array}\\[6mm]
0, &\text{otherwise}.
\end{cases}
\end{equation}
Here $\bar\eta_k\in(0,1]$ and $\nu_{\rm qn}\in(0,2)$ are defined in Section~\ref{sec:notation-parameters}; $\nu_{\rm qn}=1$ is the reference choice.  If \eqref{eq:predictor-eta} returns zero at index $j$, the rollout stops: for every remaining index $\ell>j$ we set $\widetilde x_{k,\ell}=\widetilde x_{k,j}$ and $\widetilde r_{k,\ell}=\widetilde r_{k,j}$, so the endpoint notation $\widetilde x_{k,K}$ and $\widetilde r_{k,K}$ remains well defined.

The rule \eqref{eq:predictor-eta} is not a line search.  It uses only dot products and the action $B_kd_{k,j}$, which is already needed for the modeled-residual update.  In particular, it evaluates neither $f$ nor $\grad f$ at any intermediate pseudo-point.  Moreover, whenever $\eta_{k,j}>0$,
\begin{equation}
\label{eq:pseudo-model-decrease}
\begin{aligned}
\mathcal Q_{k,j}(\eta_{k,j}d_{k,j})
&=\eta_{k,j}\ip{\widetilde r_{k,j}}{d_{k,j}}
  +\frac{\eta_{k,j}^2}{2}\ip{d_{k,j}}{B_kd_{k,j}}\\
&\le
\eta_{k,j}\ip{\widetilde r_{k,j}}{d_{k,j}}
\left(1-\frac{\nu_{\rm qn}}{2}\right)<0.
\end{aligned}
\end{equation}
Thus every accepted pseudo-step decreases the current quadratic model.  This is only a model certificate because $\widetilde r_{k,j}$ is modeled after $j=0$.  Actual descent for $\phi_k$ is checked once, at the final point, by \eqref{eq:armijo-certificate}.

The frozen $K$-step predictor is then
\begin{align}
    \widetilde x_{k,j+1}&=\widetilde x_{k,j}+\eta_{k,j}d_{k,j},
    \label{eq:model-rollout-state}\\
    \widetilde r_{k,j+1}&=\widetilde r_{k,j}
       +B_k(\widetilde x_{k,j+1}-\widetilde x_{k,j}),
    \qquad j=0,\ldots,K-1.
    \label{eq:model-rollout-residual}
\end{align}
Only the final point
\begin{equation}
\label{eq:predictor-endpoint}
    \widehat x_k:=\widetilde x_{k,K}
\end{equation}
receives a new true gradient.  For $K=0$, the convention in Section~\ref{sec:notation-parameters} applies.

\begin{proposition}[Frozen-rollout cancellation]
\label{prop:frozen-cancellation-new}
Assume $B_kH_k=I$ and that no safeguard changes the directions in \eqref{eq:model-rollout-direction}.  Then
\begin{equation}
\label{eq:frozen-residual-cancel-new}
    \widetilde r_{k,j+1}=(1-\eta_{k,j})\widetilde r_{k,j}
\end{equation}
and
\begin{equation}
\label{eq:frozen-endpoint-new}
    \widehat x_k
    =a_{k-1}-\eta_k^{\rm eff}H_k\bar r_k(a_{k-1}),
    \qquad
    \eta_k^{\rm eff}:=1-\prod_{j=0}^{K-1}(1-\eta_{k,j}).
\end{equation}
If the models are paired exactly, then the quotient in \eqref{eq:predictor-eta} is one; hence, for $\nu_{\rm qn}=1$, every nonzero pseudo-step has $\eta_{k,j}=\bar\eta_k$.
\end{proposition}

\begin{proof}
Using $B_kH_k=I$ in \eqref{eq:model-rollout-residual} gives
$\widetilde r_{k,j+1}=(I-\eta_{k,j}B_kH_k)\widetilde r_{k,j}$, which is \eqref{eq:frozen-residual-cancel-new}.  Summing the displacements in \eqref{eq:model-rollout-state} gives the telescoping coefficient in \eqref{eq:frozen-endpoint-new}.  Finally, $B_kd_{k,j}=-\widetilde r_{k,j}$ under exact pairing, so
$\ip{d_{k,j}}{B_kd_{k,j}}=-\ip{\widetilde r_{k,j}}{d_{k,j}}$.
\end{proof}

Thus a frozen $K$-deep predictor does not create $K$ independent Newton directions.  It realizes one direction with a larger effective damping.  This identity is exact and is used in both the global and local analysis.

\subsection{Paired BFGS models}
\label{sec:bfgs-model}

Let $H\succ0$, $B=H^{-1}$, and let $(s,y)$ satisfy $\ip{s}{y}>0$.  Set
\[
  \omega:=(\ip{s}{y})^{-1},
  \qquad
  V:=I-\omega sy^\top.
\]
The inverse and direct BFGS updates are
\begin{align}
    H^+&=VHV^\top+\omega ss^\top,
    \label{eq:inverse-bfgs}\\
    B^+&=B-\frac{Bss^\top B}{\ip{s}{Bs}}
           +\frac{yy^\top}{\ip{s}{y}}.
    \label{eq:direct-bfgs}
\end{align}
They satisfy $H^+y=s$, $B^+s=y$, and $B^+=(H^+)^{-1}$.  The scalar $\omega$ is local to this update and is unrelated to the DRS contraction factor $\rho$ defined later in \eqref{eq:global-delta-rho}.

The reference memory accepts an exact pair only if
\begin{equation}
\label{eq:pair-curvature-threshold}
  \ip{s}{y}\ge
  \epsilon_{\rm curv}\norm{s}\norm{y},
  \qquad \epsilon_{\rm curv}\in[0,1).
\end{equation}
In exact arithmetic, positivity follows from \eqref{eq:automatic-curvature}; \eqref{eq:pair-curvature-threshold} is a finite-precision and conditioning safeguard.  An unaccepted pair is skipped.  Starting from the paired scalar models defined in \eqref{eq:initial-qn-scaling}, the accepted pairs in $\mathcal M_k$ are replayed in chronological order using \eqref{eq:inverse-bfgs}--\eqref{eq:direct-bfgs}.

For large problems, the inverse action is implemented by the L-BFGS two-loop recursion \cite{LiuNocedal1989,ByrdNocedalSchnabel1994}.  The matching direct action can be obtained by replaying \eqref{eq:direct-bfgs} on a vector.  The local superlinear result in Section~\ref{sec:superlinear} applies directly to full-memory BFGS and conditionally to limited-memory variants through the transported Dennis--Mor\'e condition.

\subsection{Paired safeguarded SR1 models}
\label{sec:sr1-model}

The symmetric-rank-one update uses one rank-one correction.  Define
\begin{equation}
\label{eq:sr1-u-v}
    u=s-Hy,
    \qquad
    v=y-Bs.
\end{equation}
When the denominators are nonzero,
\begin{align}
    H^+&=H+\frac{uu^\top}{\ip{u}{y}},
    \label{eq:inverse-sr1}\\
    B^+&=B+\frac{vv^\top}{\ip{v}{s}}.
    \label{eq:direct-sr1}
\end{align}
Both updates satisfy the secant equation.  If $B=H^{-1}$, they remain paired whenever both formulas are defined; this follows from the Sherman--Morrison identity.

Standard SR1 may be indefinite.  Choose $\epsilon_{\rm sr1}\in(0,1)$ and the spectral bounds $0<\underline h\le\overline h<\infty$ from Section~\ref{sec:notation-parameters}.  The certified CR-DRS realization accepts \eqref{eq:inverse-sr1}--\eqref{eq:direct-sr1} only when
\begin{equation}
\label{eq:sr1-skip}
    |\ip{u}{y}|\ge\epsilon_{\rm sr1}\norm{u}\norm{y}
\end{equation}
and
\begin{equation}
\label{eq:sr1-spectral-bounds}
   \underline h I\preceq H^+\preceq\overline h I,
   \qquad
   \overline h^{-1}I\preceq B^+\preceq\underline h^{-1}I.
\end{equation}
If either test fails, the reference algorithm skips the pair; a restart to the scalar models in \eqref{eq:initial-qn-scaling} is an equivalent safe option.  These tests keep the rank-one curvature model positive definite and its search directions descending.

\begin{proposition}[Paired update property]
\label{prop:paired-update-property}
Suppose $B=H^{-1}$.  BFGS updates \eqref{eq:inverse-bfgs}--\eqref{eq:direct-bfgs} preserve $B^+=(H^+)^{-1}$.  SR1 updates \eqref{eq:inverse-sr1}--\eqref{eq:direct-sr1} also preserve this identity whenever their denominators are nonzero.
\end{proposition}

\begin{proof}
The BFGS statement is the standard inverse/direct equivalence.  For SR1, apply Sherman--Morrison to $H^+=H+uu^\top/\ip{u}{y}$.  Since $Bu=Bs-y=-v$ and
$\ip{u}{y}+\ip{u}{Bu}=-\ip{v}{s}$, the resulting inverse is exactly \eqref{eq:direct-sr1}.
\end{proof}

Therefore the exact-arithmetic analysis below uses $H_kB_k=I$ as an identity of the algorithm, not as an asymptotic hypothesis.  Any finite-precision loss of pairing is handled as an implementation diagnostic in \eqref{eq:BH-consistency}.

\subsection{Certified fixed-budget CR-DRS}
\label{sec:certified-algorithm}

For the convergence theorem, assume that $f$ is $\sigma$-strongly convex and $\beta$-smooth.  Then every $\phi_k$ is
\begin{equation}
\label{eq:m-L-prox}
    m\text{-strongly convex and }L\text{-smooth},
    \qquad
    m:=1+\gamma\sigma,
    \quad
    L:=1+\gamma\beta.
\end{equation}
Choose $c_A\in(0,1/2)$, $\theta_Q\in(0,1)$, and an integer $b\ge1$.  A quasi-Newton candidate is accepted if
\begin{align}
    \phi_k(\widehat x_k)
    &\le \phi_k(a_{k-1})
      +c_A\ip{\bar r_k(a_{k-1})}{\widehat x_k-a_{k-1}},
    \label{eq:armijo-certificate}\\
    \norm{r_k(\widehat x_k)}
    &\le \frac{m}{L}\theta_Q
       \norm{\bar r_k(a_{k-1})}.
    \label{eq:residual-certificate}
\end{align}
The candidate is considered only when
\begin{equation}
\label{eq:descent-direction-certificate}
    \ip{\bar r_k(a_{k-1})}{\widehat x_k-a_{k-1}}<0.
\end{equation}
The direction test and the first inequality certify descent.  The second inequality certifies contraction of the unknown proximal-point error.  These are one-shot endpoint tests, not a line search: the stored value $f(a_{k-1})$ gives $\phi_k(a_{k-1})$, and the single value-gradient evaluation at $\widehat x_k$ gives every other quantity in \eqref{eq:armijo-certificate}--\eqref{eq:residual-certificate}.

If any of the three tests fails, the method discards the candidate from the curvature memory and applies $b$ safe gradient steps to the current proximal objective:
\begin{equation}
\label{eq:safe-gradient-step}
    u_{j+1}=u_j-\tau_G r_k(u_j),
    \qquad
    \tau_G:=\frac{2}{m+L},
    \qquad
    u_0=a_{k-1}.
\end{equation}
The residual at $u_0$ is known from transport.  Each new $u_{j+1}$ receives one gradient.  Thus an accepted quasi-Newton candidate costs one new gradient; a rejected candidate followed by the fallback costs at most $b+1$.

\begin{algorithm}[t]
\caption{Certified CR-DRS with BFGS or safeguarded SR1}
\label{alg:certified-crdrs}
\begin{algorithmic}[1]
\REQUIRE $z_0$; exact anchor $(a_{-1},f(a_{-1}),\grad f(a_{-1}))$; initial memory $\mathcal M_0$; update type (BFGS or SR1); fixed parameters $K,b,m_q,\gamma,\alpha,c_A,\theta_Q,\nu_{\rm qn},\bar\eta,\epsilon_{\rm curv},\epsilon_{\rm sr1},\underline h,\overline h,\vartheta_0$ from Section~\ref{sec:notation-parameters}.
\FOR{$k=0,1,2,\ldots$}
\STATE Transport $\bar r_k=\gamma\grad f(a_{k-1})+a_{k-1}-z_k$ and set the reference cap $\bar\eta_k\leftarrow\bar\eta$.
\STATE Build paired $(B_k,H_k)$ from $\mathcal M_k$ by the selected BFGS or safeguarded SR1 update.
\STATE Generate $\widehat x_k$ by \eqref{eq:model-rollout-initial}--\eqref{eq:predictor-endpoint}.
\STATE Evaluate $f(\widehat x_k)$ and $\grad f(\widehat x_k)$; form $r_k(\widehat x_k)$.
\IF{\eqref{eq:descent-direction-certificate}--\eqref{eq:residual-certificate} hold}
    \STATE Set $x_k=\widehat x_k$ and keep its exact value and gradient.
\ELSE
    \STATE Starting from $u_0=a_{k-1}$, apply $b$ steps of \eqref{eq:safe-gradient-step}, evaluating one new gradient at each endpoint; set $x_k=u_b$.
\ENDIF
\STATE Set $\mathcal M_{k+1}\leftarrow\mathcal M_k$ and form the exact pair
$s_k=x_k-a_{k-1}$,
$q_k=r_k(x_k)-\bar r_k$.  Append $(s_k,q_k)$ only if \eqref{eq:pair-curvature-threshold} and the chosen BFGS or SR1 safeguards hold; in the limited-memory case, discard the oldest pair if the memory would exceed $m_q$.
\STATE Set $y_k^{g}=P_{\gamma g}(2x_k-z_k)$ and
$z_{k+1}=z_k+2\alpha(y_k^{g}-x_k)$.
\STATE Set the next exact anchor $a_k=x_k$.
\ENDFOR
\end{algorithmic}
\end{algorithm}

Only exact-gradient endpoints enter $\mathcal M_k$.  The failed model candidate is not inserted.  This rule prevents the quasi-Newton memory from learning its own synthetic curvature.

\section{Global convergence with a fixed inner budget}
\label{sec:global}

We state the main result first.  Its proof is then split into three short estimates.

\begin{theorem}[Global linear convergence of certified CR-DRS]
\label{thm:global-certified-crdrs}
Assume that $f$ is $\sigma$-strongly convex and $\beta$-smooth, with $0<\sigma\le\beta$, and that $g$ is proper, closed, and convex.  Fix $\gamma>0$ and let
\begin{equation}
\label{eq:global-delta-rho}
    \delta:=\max\left\{
      \frac{\gamma\beta-1}{\gamma\beta+1},
      \frac{1-\gamma\sigma}{1+\gamma\sigma}
    \right\},
    \qquad
    \rho:=|1-\alpha|+\alpha\delta<1.
\end{equation}
Run Algorithm~\ref{alg:certified-crdrs}.  Define
\begin{equation}
\label{eq:inner-theta}
    \theta_G:=\frac{L-m}{L+m},
    \qquad
    \theta:=\max\{\theta_Q,\theta_G^b\},
    \qquad
    L_p:=\frac1m,
    \qquad
    c:=2\alpha,
\end{equation}
where $m,L$ are given by \eqref{eq:m-L-prox}.  If
\begin{equation}
\label{eq:small-gain-scalar}
    \theta
    <\frac{1-\rho}
    {(1-\rho)(1+cL_p)+cL_p(1+\rho)},
\end{equation}
then the DRS state $z_k$ converges $R$-linearly to the unique fixed point $z_\star$, and the approximate proximal points $x_k$ converge $R$-linearly to the unique minimizer
$x_\star=P_{\gamma f}z_\star$ of \eqref{eq:main-problem}.

More precisely, with
\begin{equation}
\label{eq:global-errors}
    A_k:=\norm{z_k-z_\star},
    \qquad
    E_k:=\norm{x_k-P_{\gamma f}z_k},
\end{equation}
we have the componentwise recursion
\begin{equation}
\label{eq:global-matrix-recursion}
    \begin{bmatrix}A_{k+1}\\E_{k+1}\end{bmatrix}
    \le
    M
    \begin{bmatrix}A_k\\E_k\end{bmatrix},
    \qquad
    M:=
    \begin{bmatrix}
       \rho & c\\
       \theta L_p(1+\rho) & \theta(1+cL_p)
    \end{bmatrix},
\end{equation}
where \eqref{eq:small-gain-scalar} is equivalent to $\varrho(M)<1$.
\end{theorem}

\subsection{Proof step 1: the inner call is contractive}

\begin{lemma}[Certified inner contraction]
\label{lem:certified-inner-contraction}
Let $p_k=P_{\gamma f}z_k$.  One call of the certified inner routine satisfies
\begin{equation}
\label{eq:inner-contraction-result}
    \norm{x_k-p_k}
    \le\theta\norm{a_{k-1}-p_k},
\end{equation}
with $\theta$ from \eqref{eq:inner-theta}.
\end{lemma}

\begin{proof}
If the quasi-Newton candidate is accepted, strong convexity and smoothness give
\[
    \norm{\widehat x_k-p_k}
    \le\frac1m\norm{r_k(\widehat x_k)}
    \le\frac{\theta_Q}{L}\norm{r_k(a_{k-1})}
    \le\theta_Q\norm{a_{k-1}-p_k}.
\]
If the candidate is rejected, apply the sector inequality from Lemma~\ref{lem:sector} to $r_k=\grad\phi_k$.  For one step $u^+=u-2r_k(u)/(m+L)$,
\[
\begin{aligned}
    \norm{u^+-p_k}^2
    &=\norm{u-p_k}^2
      -\frac{4}{m+L}\ip{u-p_k}{r_k(u)}
      +\frac{4}{(m+L)^2}\norm{r_k(u)}^2\\
    &\le\left(\frac{L-m}{L+m}\right)^2\norm{u-p_k}^2.
\end{aligned}
\]
After $b$ steps the factor is $\theta_G^b$.  Taking the larger of the two cases proves \eqref{eq:inner-contraction-result}.  The descent lemma also gives
\[
    \phi_k(u^+)\le\phi_k(u)
       -\frac{2m}{(m+L)^2}\norm{r_k(u)}^2,
\]
so every fallback step has a computable strict-descent certificate unless $u=p_k$.
\end{proof}

\subsection{Proof step 2: an inexact smooth prox perturbs one DRS step linearly}

\begin{lemma}[Outer perturbation]
\label{lem:outer-perturbation-focused}
Let $p_k=P_{\gamma f}z_k$ and write $x_k=p_k+e_k$.  If the $g$-proximal map is evaluated exactly, then
\begin{equation}
\label{eq:outer-perturb-focused}
    \norm{z_{k+1}-Tz_k}\le 2\alpha\norm{e_k},
\end{equation}
where $T=(1-\alpha)I+\alpha R_{\gamma g}R_{\gamma f}$ is the exact DRS map.  Consequently,
\begin{equation}
\label{eq:outer-error-focused}
    A_{k+1}\le\rho A_k+cE_k.
\end{equation}
\end{lemma}

\begin{proof}
The approximate reflected $f$-proximal output differs from the exact one by $2e_k$.  Since $R_{\gamma g}$ is nonexpansive by Lemma~\ref{lem:reflected-nonexpansive}, the relaxed outer map changes by at most $2\alpha\norm{e_k}$.  Combining this with the contraction estimate
$\norm{Tz_k-Tz_\star}\le\rho\norm{z_k-z_\star}$ from Theorem~\ref{thm:exact-linear-DR} gives \eqref{eq:outer-error-focused}.
\end{proof}

\subsection{Proof step 3: the moving proximal point gives a two-state recursion}

\begin{lemma}[Tracking recursion]
\label{lem:tracking-recursion-focused}
Under the assumptions of Theorem~\ref{thm:global-certified-crdrs},
\begin{equation}
\label{eq:tracking-error-focused}
    E_{k+1}
    \le \theta L_p(1+\rho)A_k
       +\theta(1+cL_p)E_k.
\end{equation}
\end{lemma}

\begin{proof}
The exact proximal map is $L_p$-Lipschitz by \eqref{eq:strong-prox-Lipschitz}.  Apply Lemma~\ref{lem:certified-inner-contraction} to the next call:
\[
\begin{aligned}
    E_{k+1}
    &\le\theta\norm{x_k-P_{\gamma f}z_{k+1}}\\
    &\le\theta E_k+\theta L_p\norm{z_{k+1}-z_k}\\
    &\le\theta E_k+\theta L_p(A_{k+1}+A_k).
\end{aligned}
\]
Insert \eqref{eq:outer-error-focused} to obtain \eqref{eq:tracking-error-focused}.
\end{proof}

\begin{proof}[Proof of Theorem~\ref{thm:global-certified-crdrs}]
Lemmas~\ref{lem:outer-perturbation-focused} and~\ref{lem:tracking-recursion-focused} give \eqref{eq:global-matrix-recursion}.
The matrix $M$ is nonnegative.  Since $\rho<1$, the condition $\varrho(M)<1$ is equivalent here to
\[
    \det(I-M)
    =(1-\rho)\bigl[1-\theta(1+cL_p)\bigr]
      -c\theta L_p(1+\rho)>0,
\]
which is exactly \eqref{eq:small-gain-scalar}.  Hence $M^k\to0$ linearly.  The convergence of $A_k$ and $E_k$ follows.  Finally,
\[
    \norm{x_k-x_\star}
    \le E_k+\norm{P_{\gamma f}z_k-P_{\gamma f}z_\star}
    \le E_k+L_pA_k.
\]
\end{proof}

\begin{corollary}[Gradient count]
\label{cor:global-gradient-count}
Under Theorem~\ref{thm:global-certified-crdrs}, there are constants $C>0$ and $\widehat\rho\in(\varrho(M),1)$ such that
\[
    A_k+E_k\le C\widehat\rho^k.
\]
For a prescribed target $\varepsilon>0$, to reach $A_k+E_k\le\varepsilon$, Algorithm~\ref{alg:certified-crdrs} uses at most
\begin{equation}
\label{eq:fixed-gradient-count}
    (b+1)\left\lceil
    \frac{\log(C/\varepsilon)}{\log(1/\widehat\rho)}
    \right\rceil+O(1)
\end{equation}
new smooth-gradient evaluations.  If every quasi-Newton candidate is accepted, the factor $b+1$ is replaced by $1$.
\end{corollary}

\section{Local superlinear behavior of transported quasi-Newton steps}
\label{sec:superlinear}

The global theorem uses only the acceptance certificate.  We now use the fact that the candidate is quasi-Newton.  The main theorem again comes first.

\begin{theorem}[Transported Dennis--Mor\'e theorem]
\label{thm:transported-dm-focused}
Let $z_\star$ be the exact DRS fixed point from Theorem~\ref{thm:global-certified-crdrs} and set $p_\star:=P_{\gamma f}z_\star$.  For each $k$, let
\[
    p_k=P_{\gamma f}z_k,
    \qquad
    d_k:=a_{k-1}-p_k,
    \qquad
    \overline G_k:=\int_0^1
      \grad^2\phi_k(p_k+td_k)\,dt.
\]
Assume $d_k\to0$, $d_k\ne0$, $a_{k-1}\to p_\star$, $p_k\to p_\star$, and $\grad^2 f$ is continuous near $p_\star$.  Write the actual transported candidate as
\begin{equation}
\label{eq:actual-transported-focused}
    \widehat x_k
    =a_{k-1}-\eta_k^{\rm eff}H_k r_k(a_{k-1})+\xi_k.
\end{equation}
Assume the direct and inverse models are paired as in Proposition~\ref{prop:paired-update-property}, so $B_k=H_k^{-1}$ in exact arithmetic.  Here $\eta_k^{\rm eff}$ is the effective damping from \eqref{eq:frozen-endpoint-new} (or its value after any explicit cap), and $\xi_k$ is the remaining displacement change caused by clipping or truncation.  Suppose
\begin{align}
    \sup_k\norm{H_k}&<\infty,
    &\eta_k^{\rm eff}&\to1,
    \label{eq:dm-basic-focused}\\
    \frac{\norm{(B_k-\overline G_k)d_k}}{\norm{d_k}}&\to0,
    &
    \frac{\norm{\xi_k}}{\norm{d_k}}&\to0.
    \label{eq:dm-directional-focused}
\end{align}
Then
\begin{equation}
\label{eq:transport-superlinear-focused}
    \frac{\norm{\widehat x_k-p_k}}
         {\norm{a_{k-1}-p_k}}
    \longrightarrow0.
\end{equation}
Moreover, for every fixed $\theta_Q\in(0,1)$ and $c_A\in(0,1/2)$, the descent-direction test \eqref{eq:descent-direction-certificate}, the residual test \eqref{eq:residual-certificate} and Armijo test \eqref{eq:armijo-certificate} hold for all sufficiently large $k$.  Hence Algorithm~\ref{alg:certified-crdrs} eventually accepts the transported candidate with one new gradient and does not call the fallback block.
\end{theorem}

\begin{proof}
Since $r_k(p_k)=0$, the fundamental theorem of calculus gives
\begin{equation}
\label{eq:mean-hessian-focused}
    r_k(a_{k-1})=\overline G_kd_k.
\end{equation}
Insert this into \eqref{eq:actual-transported-focused}.  Exact pairing gives $H_kB_k=I$, and therefore
\[
\begin{aligned}
    \widehat x_k-p_k
    ={}&(1-\eta_k^{\rm eff})d_k
      +\eta_k^{\rm eff}H_k(B_k-\overline G_k)d_k+\xi_k.
\end{aligned}
\]
Divide by $\norm{d_k}$ and use \eqref{eq:dm-basic-focused}--\eqref{eq:dm-directional-focused}.  This proves \eqref{eq:transport-superlinear-focused}.

For the residual test, local smoothness and strong convexity give
\[
    \frac{\norm{r_k(\widehat x_k)}}{\norm{r_k(a_{k-1})}}
    \le \frac{L}{m}
       \frac{\norm{\widehat x_k-p_k}}{\norm{d_k}}
    \longrightarrow0.
\]
Thus \eqref{eq:residual-certificate} eventually holds.

Set $\widehat d_k=\widehat x_k-p_k=o(\norm{d_k})$ and
$\widehat s_k=\widehat x_k-a_{k-1}=-d_k+\widehat d_k$.  Continuity of the Hessian gives the local expansions
\[
\begin{aligned}
    \phi_k(a_{k-1})-\phi_k(p_k)
      &=\frac12\ip{d_k}{G_\star d_k}+o(\norm{d_k}^2),\\
    \ip{r_k(a_{k-1})}{\widehat s_k}
      &=-\ip{d_k}{G_\star d_k}+o(\norm{d_k}^2),\\
    \phi_k(\widehat x_k)-\phi_k(p_k)
      &=o(\norm{d_k}^2),
\end{aligned}
\]
where $G_\star=I+\gamma\grad^2f(p_\star)\succ0$.  Therefore
\[
\begin{aligned}
&\phi_k(\widehat x_k)-\phi_k(a_{k-1})
   -c_A\ip{r_k(a_{k-1})}{\widehat s_k}\\
&\qquad
=-(1/2-c_A)\ip{d_k}{G_\star d_k}
  +o(\norm{d_k}^2)<0
\end{aligned}
\]
for all sufficiently large $k$.  This is the Armijo test.  The same expansion gives $\ip{r_k(a_{k-1})}{\widehat s_k}<0$ for all sufficiently large $k$, so \eqref{eq:descent-direction-certificate} also holds.
\end{proof}

The theorem isolates the new moving-target condition: $B_k$ must be accurate in the direction from the old exact-gradient state to the minimizer of the new proximal problem.  A classical secant condition on an unrelated historical direction is not enough.

\subsection{Why the exact proximal path is predictable}

\begin{proposition}[Second-order motion of the proximal point]
\label{prop:prox-path-focused}
Let $p(z)=P_{\gamma f}z$ and let
$G(z)=I+\gamma\grad^2f(p(z))$.  Suppose $G(z)\succeq mI$ and $\grad^2f$ is locally Lipschitz.  Then, for small $\Delta z$,
\begin{equation}
\label{eq:prox-path-focused}
    p(z+\Delta z)
    =p(z)+G(z)^{-1}\Delta z+O(\norm{\Delta z}^2).
\end{equation}
\end{proposition}

\begin{proof}
The optimality equation is
$\gamma\grad f(p(z))+p(z)=z$.  Taylor expansion at $p(z)$ gives
\[
    \Delta z=G(z)\bigl[p(z+\Delta z)-p(z)\bigr]
       +O\bigl(\norm{p(z+\Delta z)-p(z)}^2\bigr).
\]
The proximal map is $1/m$-Lipschitz, so the remainder is $O(\norm{\Delta z}^2)$.  Multiplication by $G(z)^{-1}$ gives \eqref{eq:prox-path-focused}.
\end{proof}

At an exact old proximal point, the new residual is simply $z_{k-1}-z_k$.  Thus a full inverse-Hessian transport step is the first-order predictor in \eqref{eq:prox-path-focused}; exact inverse curvature gives a second-order tracking error.

\subsection{BFGS and SR1 specialization}

\begin{corollary}[Full-memory BFGS or SR1]
\label{cor:bfgs-sr1-superlinear}
Assume the hypotheses of Theorem~\ref{thm:transported-dm-focused}.  Suppose the BFGS or safeguarded SR1 realization has entered a local regime in which
\begin{enumerate}[label=(\alph*),leftmargin=2.2em]
    \item the direct and inverse actions remain paired;
    \item the effective damping becomes full, so $\eta_k^{\rm eff}\to1$ and $\xi_k=o(\norm{d_k})$;
    \item the update satisfies the transported Dennis--Mor\'e condition
    \[
       \frac{\norm{(B_k-\overline G_k)d_k}}{\norm{d_k}}\to0.
    \]
\end{enumerate}
Then the accepted CR-DRS proximal candidate is $Q$-superlinear relative to the ordinary warm start and uses one new gradient for all sufficiently large $k$.

A sufficient stronger condition is $\norm{B_k-G_\star}\to0$, where
$G_\star=I+\gamma\grad^2f(p_\star)$.  Classical full-memory BFGS and SR1 theory supplies Dennis--Mor\'e conditions under the usual local smoothness, nonsingularity, and step-acceptance assumptions \cite{BroydenDennisMore1973,dennismore1974,DennisMore1977,NocedalWright2006}.  Generic fixed-memory L-BFGS does not automatically have this property; for that realization the displayed transported condition is an explicit additional hypothesis.
\end{corollary}

\section{Justifying the transported Dennis--Mor\'e condition}
\label{sec:dm-coverage}

The transported Dennis--Mor\'e condition is the main local assumption that is specific to CR-DRS.  Classical quasi-Newton theory learns curvature from steps generated by one fixed objective.  Here the direction
\[
    d_k=a_{k-1}-p_k,
    \qquad p_k=P_{\gamma f}z_k,
\]
also contains the motion of the exact proximal point caused by the outer DRS update.  The affine change of right-hand side does not alter any exact secant, but it determines the direction in which the stored Hessian model is queried.

The direct and inverse models used in this paper are paired.  Proposition~\ref{prop:paired-update-property} gives $B_k=H_k^{-1}$ in exact arithmetic for both BFGS and accepted SR1 updates.  Thus $H_kB_k=I$ is built into Theorem~\ref{thm:transported-dm-focused}; no separate inverse-consistency limit is required.  The substantive condition is
\begin{equation}
\label{eq:dm-real-obstacle}
    \frac{\norm{(B_k-\overline G_k)d_k}}{\norm{d_k}}\longrightarrow0.
\end{equation}

Set
\begin{equation}
\label{eq:dm-coverage-Gstar}
    G_\star:=I+\gamma\grad^2f(p_\star),
    \qquad
    u_k:=\frac{d_k}{\norm{d_k}}
    \quad(d_k\ne0).
\end{equation}
Under the hypotheses of Theorem~\ref{thm:transported-dm-focused}, continuity of $\grad^2f$ gives $\norm{\overline G_k-G_\star}\to0$.  Hence
\begin{equation}
\label{eq:dm-equivalent-fixed-Hessian}
    \norm{(B_k-\overline G_k)u_k}\to0
    \quad\Longleftrightarrow\quad
    \norm{(B_k-G_\star)u_k}\to0.
\end{equation}
The moving-target question is therefore a directional approximation problem for one limiting Hessian $G_\star$.

\subsection{How the outer step enters the transported direction}

Let
\[
    e_{k-1}:=a_{k-1}-p_{k-1},
    \qquad
    \Delta z_k:=z_k-z_{k-1},
    \qquad
    G_{k-1}:=I+\gamma\grad^2f(p_{k-1}).
\]
Proposition~\ref{prop:prox-path-focused} gives
\begin{equation}
\label{eq:dm-direction-decomposition}
    d_k
    =e_{k-1}-G_{k-1}^{-1}\Delta z_k
      +O(\norm{\Delta z_k}^2).
\end{equation}
The first term is the previous smooth-proximal error.  The second is the first-order response of the exact proximal point to the new DRS right-hand side.  The right-hand side cancels from the reusable secant identity \eqref{eq:reusable-secant}; nevertheless, it can request curvature in a direction not yet represented in the memory.

\begin{remark}[A DRS-realizable obstruction]
\label{rem:dm-orthogonal-counterexample}
The obstruction is not an artifact of allowing an arbitrary sequence of centers.  Let $\Hcal=\R^2$, let $e_1,e_2$ be the coordinate vectors, and choose
\[
    G=\diag(g_1,g_2),
    \qquad g_1,g_2>1,
    \qquad
    f(x)=\frac{1}{2\gamma}x^\top(G-I)x.
\]
Then
\[
    P_{\gamma f}z=G^{-1}z,
    \qquad
    \phi_z(x)=\frac12x^\top Gx-z^\top x+\frac12\norm z^2.
\]
Fix a DRS state $\bar z$ and write $\bar p=G^{-1}\bar z$.  Choose the nonsmooth term
\[
    g=\iota_{\{\bar y\}},
    \qquad
    \bar y:=\bar p+\frac{t}{2\alpha}e_2,
    \qquad t\ne0.
\]
Its proximal map is the constant map $P_{\gamma g}\equiv\bar y$.  The corresponding DRS fixed point is
\[
    z_\star=G\bar y
      =\bar z+\frac{g_2t}{2\alpha}e_2,
\]
so the construction can be placed in an arbitrarily small neighborhood of a fixed point by taking $|t|$ small.  The valid DRS update from $\bar z$ is
\[
    \Delta z=2\alpha(\bar y-\bar p)=t e_2.
\]
If the old anchor is exact, then $d=-G^{-1}\Delta z=-(t/g_2)e_2$.  Suppose the available secants lie in $\operatorname{span}\{e_1\}$ and
$B=\diag(g_1,b)$ with $b\ne g_2$.  Then
\[
    \frac{\norm{(B-G)d}}{\norm d}=|b-g_2|.
\]
Thus the DRS structure alone does not imply transported Dennis--Mor\'e accuracy from past secants.

This example is deliberately local.  It shows that one DRS step can introduce a direction that has not yet been learned.  It does \emph{not} show that DRS keeps producing mutually orthogonal innovations, or that a full-memory model cannot converge on the increasing span of the observed directions.  The latter question is addressed by the active-subspace conditions below.
\end{remark}

\subsection{Deterministic sufficient conditions}

The strongest sufficient condition is immediate.

\begin{corollary}[Full-model convergence]
\label{cor:dm-full-model-convergence}
If $\norm{B_k-G_\star}\to0$, then the transported Dennis--Mor\'e condition holds for every sequence of nonzero directions $d_k\to0$, independently of how the outer DRS step chooses those directions.
\end{corollary}

Full operator convergence is stronger than necessary.  It is enough to learn the subspace visited by the transported directions.

\begin{proposition}[Active-subspace coverage]
\label{prop:dm-active-subspace}
Let $\mathcal U_k\subseteq\Hcal$ be linear subspaces and let $\Pi_k$ be their orthogonal projectors.  Suppose
\begin{align}
    \sup_k\norm{B_k-G_\star}&<\infty,
    \label{eq:dm-subspace-bounded}\\
    \frac{\norm{(I-\Pi_k)d_k}}{\norm{d_k}}&\to0,
    \label{eq:dm-subspace-direction}\\
    \norm{(B_k-G_\star)\Pi_k}&\to0.
    \label{eq:dm-subspace-model}
\end{align}
Then \eqref{eq:dm-real-obstacle} holds.
\end{proposition}

\begin{proof}
Write
\[
    (B_k-G_\star)u_k
    =(B_k-G_\star)\Pi_ku_k
      +(B_k-G_\star)(I-\Pi_k)u_k.
\]
The first term tends to zero by \eqref{eq:dm-subspace-model}.  The second tends to zero by \eqref{eq:dm-subspace-bounded}--\eqref{eq:dm-subspace-direction}.  Equation~\eqref{eq:dm-equivalent-fixed-Hessian} completes the proof.
\end{proof}

In particular, one may take a nested chain $\mathcal U_1\subseteq\mathcal U_2\subseteq\cdots$ that contains the observed directions.  If the restricted error $\norm{(B_k-G_\star)\Pi_k}$ tends to zero, then widening of the span is harmless.  Remark~\ref{rem:dm-orthogonal-counterexample} only shows that this restricted convergence is not automatic before a newly requested direction has entered the learned subspace.

\subsection{When the nonsmooth term makes the forcing nonadversarial}
\label{sec:dm-g-structure}

The outer correction is endogenous: it is generated by the same local DRS map at every iteration.  To use this fact, one must understand when the nonsmooth proximal map has a stable local derivative.  The point at which this derivative is taken is important.  Let
\begin{equation}
\label{eq:dm-prox-graph-point}
    p_\star=P_{\gamma f}z_\star,
    \qquad
    w_\star:=R_{\gamma f}z_\star=2p_\star-z_\star.
\end{equation}
The smooth proximal optimality condition gives
\(
 z_\star-p_\star=\gamma\grad f(p_\star)
\), and therefore
\begin{equation}
\label{eq:dm-wstar-primal-dual}
    w_\star=p_\star-\gamma\grad f(p_\star),
    \qquad
    v_\star:=\frac{w_\star-p_\star}{\gamma}
              =-\grad f(p_\star)
              \in\subd g(p_\star).
\end{equation}
Thus differentiability of \(P_{\gamma g}\) at \(w_\star\) is a property of the primal--dual graph point
\(
 (p_\star,v_\star)\in\operatorname{gra}\subd g
\), not simply a smoothness property of \(g\) at \(p_\star\).  Indeed, \(g\) can be nonsmooth at \(p_\star\) while its proximal map is affine in a neighborhood of \(w_\star\).  Equivalently, since the Moreau envelope
\[
    e_{\gamma g}(w)
      :=\min_x\left\{g(x)+\frac{1}{2\gamma}\norm{x-w}^2\right\}
\]
satisfies
\begin{equation}
\label{eq:dm-moreau-gradient}
    \grad e_{\gamma g}(w)
       =\frac{1}{\gamma}\bigl(w-P_{\gamma g}w\bigr),
\end{equation}
Fr\'echet differentiability of \(P_{\gamma g}\) concerns second-order smoothness of the Moreau envelope at \(w_\star\), rather than differentiability of \(g\) at \(p_\star\); see \cite[Chs.~12 and~24]{BauschkeCombettes2017}.

The useful condition is a nondegeneracy condition at the graph point.  For the regularizers below it takes the strict-complementarity form
\begin{equation}
\label{eq:dm-ri-nondegeneracy}
    -\grad f(p_\star)=v_\star
       \in\operatorname{ri}\subd g(p_\star),
\end{equation}
where \(\operatorname{ri}\) denotes relative interior.  In partial-smoothness terminology, \eqref{eq:dm-ri-nondegeneracy} places the dual vector in the relative interior of the subdifferential face associated with the active manifold.  It is the standard condition that makes the active manifold stable under small perturbations \cite{Lewis2002,LiangFadiliPeyre2017}.  More general characterizations of continuous differentiability of proximal mappings through relative-interior conditions are developed in \cite{HangSarabi2025}.

\begin{proposition}[Proximal derivative after active-manifold identification]
\label{prop:dm-active-prox-derivative}
Suppose that \(g\) is partly smooth at \(p_\star\) relative to a locally affine manifold
\(
 \mathcal M=p_\star+\mathcal T_\star
\), that \(g|_{\mathcal M}\) is twice continuously differentiable, and that \eqref{eq:dm-ri-nondegeneracy} holds.  Let
\(
 U:\R^r\to\Hcal
\)
be an isometry onto the tangent space \(\mathcal T_\star\), so that
\(
 U^*U=I_r
\)
and
\(
 UU^*=P_{\mathcal T_\star}
\).  Define the restricted Hessian
\begin{equation}
\label{eq:dm-active-hessian-g}
    \mathcal H_{g,\star}
      :=U^*\grad^2\!\bigl(g|_{\mathcal M}\bigr)(p_\star)U.
\end{equation}
Then \(P_{\gamma g}\) is continuously Fr\'echet differentiable in a neighborhood of \(w_\star\), and
\begin{equation}
\label{eq:dm-active-prox-general-derivative}
    Q_\star:=DP_{\gamma g}(w_\star)
       =U\bigl(I_r+\gamma\mathcal H_{g,\star}\bigr)^{-1}U^*.
\end{equation}
In particular, if \(g\) is locally affine on \(\mathcal M\), then
\begin{equation}
\label{eq:dm-active-prox-projector}
    Q_\star=P_{\mathcal T_\star}.
\end{equation}
\end{proposition}

\begin{proof}
Under partial smoothness and \eqref{eq:dm-ri-nondegeneracy}, proximal points generated by inputs near \(w_\star\) remain on \(\mathcal M\); this is the local identification and sensitivity statement in \cite{Lewis2002,HangSarabi2025}.  Write a nearby point on the affine manifold as
\(
 x=p_\star+Ut
\).
The proximal optimality equation, projected onto \(\mathcal T_\star\), is
\[
   U^*\bigl(x-w+\gamma\grad(g|_{\mathcal M})(x)\bigr)=0.
\]
Its derivative with respect to \(t\) at \((p_\star,w_\star)\) is
\(
 I_r+\gamma\mathcal H_{g,\star}
\), which is positive definite because \(g|_{\mathcal M}\) is convex.  The implicit-function theorem therefore gives a continuously differentiable local solution map.  Differentiating the displayed equation with respect to \(w\) gives \eqref{eq:dm-active-prox-general-derivative}.  If the restriction is affine, then \(\mathcal H_{g,\star}=0\), which gives \eqref{eq:dm-active-prox-projector}.
\end{proof}

\subsubsection{The \texorpdfstring{$\ell_1$}{l1} norm}

Let
\begin{equation}
\label{eq:dm-l1-g}
    g(x)=\lambda\norm{x}_1,
    \qquad \lambda>0.
\end{equation}
The proximal map is componentwise soft thresholding,
\begin{equation}
\label{eq:dm-l1-soft-threshold}
   \bigl(P_{\gamma g}w\bigr)_i
      =S_{\gamma\lambda}(w_i)
      :=\operatorname{sign}(w_i)
         \max\{|w_i|-\gamma\lambda,0\}.
\end{equation}
The scalar map \(S_{\gamma\lambda}\) is affine on the three open intervals separated by
\(
 \pm\gamma\lambda
\), and it is not differentiable at these two threshold values.  Hence
\begin{equation}
\label{eq:dm-l1-prox-diff-raw}
   P_{\gamma g}\text{ is Fr\'echet differentiable at }w_\star
   \quad\Longleftrightarrow\quad
   |w_{\star,i}|\ne\gamma\lambda
   \quad\text{for every }i.
\end{equation}

Let
\begin{equation}
\label{eq:dm-l1-active-sets}
   A:=\{i:p_{\star,i}\ne0\},
   \qquad
   Z:=\{i:p_{\star,i}=0\}.
\end{equation}
If \(i\in A\), stationarity gives
\(
 \grad_i f(p_\star)=-\lambda\operatorname{sign}(p_{\star,i})
\), and therefore
\[
   w_{\star,i}
     =p_{\star,i}+\gamma\lambda\operatorname{sign}(p_{\star,i}),
     \qquad
   |w_{\star,i}|>\gamma\lambda.
\]
Thus differentiability is automatic on the nonzero support.  If \(i\in Z\), then
\(
 w_{\star,i}=-\gamma\grad_i f(p_\star)
\), while the KKT condition gives
\(
 |\grad_i f(p_\star)|\le\lambda
\).  Avoiding the threshold requires the strict inequality
\begin{equation}
\label{eq:dm-l1-strict-complementarity}
   |\grad_i f(p_\star)|<\lambda
   \qquad(i\in Z).
\end{equation}
Consequently,
\begin{equation}
\label{eq:dm-l1-equivalence}
   \begin{aligned}
   P_{\gamma g}\text{ is Fr\'echet differentiable at }w_\star
   &\quad\Longleftrightarrow\quad
   -\grad f(p_\star)\in\operatorname{ri}\subd g(p_\star)\\
   &\quad\Longleftrightarrow\quad
   \eqref{eq:dm-l1-strict-complementarity}.
   \end{aligned}
\end{equation}
The identified tangent space is
\begin{equation}
\label{eq:dm-l1-tangent}
   \mathcal T_\star
     :=\{h\in\Hcal:h_i=0\text{ for every }i\in Z\},
\end{equation}
and the derivative is the coordinate projector
\begin{equation}
\label{eq:dm-l1-Q}
   Q_\star
     =\diag\bigl(\mathbf 1_{\{p_{\star,i}\ne0\}}\bigr)
     =P_{\mathcal T_\star}.
\end{equation}

The graph-point distinction is already visible in one dimension.  Let
\(
 g(x)=\lambda|x|
\)
and \(p_\star=0\).  The function \(g\) is not differentiable at \(p_\star\).  If \(w_\star=0\), however, soft thresholding is identically zero near \(w_\star\), so
\(
 DP_{\gamma g}(0)=0
\).  If instead \(w_\star=\gamma\lambda\), the same primal point lies at the boundary of two proximal cells and the proximal map is not differentiable.  The difference is the selected dual element
\(
 v_\star=(w_\star-p_\star)/\gamma\in\subd g(p_\star)
\), not the primal point alone.

\subsubsection{Anisotropic discrete total variation}

Let \(D:\Hcal\to\R^m\) be a finite-difference operator and let
\begin{equation}
\label{eq:dm-aniso-tv-g}
    g(x)=\lambda\norm{Dx}_1.
\end{equation}
Define
\begin{equation}
\label{eq:dm-aniso-tv-active-sets}
   A:=\{j:(Dp_\star)_j\ne0\},
   \qquad
   Z:=\{j:(Dp_\star)_j=0\}.
\end{equation}
The subdifferential can be represented as
\begin{equation}
\label{eq:dm-aniso-tv-subdiff}
   \subd g(p_\star)
   =\left\{\lambda D^*q:
       q_A=\operatorname{sign}(D_Ap_\star),
       \ \norm{q_Z}_\infty\le1\right\}.
\end{equation}
Thus stationarity is equivalent to the existence of a dual certificate \(q_\star\) such that
\begin{equation}
\label{eq:dm-aniso-tv-certificate}
   -\grad f(p_\star)=\lambda D^*q_\star,
   \qquad
   q_{\star,A}=\operatorname{sign}(D_Ap_\star),
   \qquad
   \norm{q_{\star,Z}}_\infty\le1.
\end{equation}
The strict-complementarity condition is the existence of such a certificate with
\begin{equation}
\label{eq:dm-aniso-tv-strict-complementarity}
   \norm{q_{\star,Z}}_\infty<1.
\end{equation}
Equivalently, \(-\grad f(p_\star)\in\operatorname{ri}\subd g(p_\star)\).  Under this condition, small perturbations preserve the zero-difference set \(Z\) and the signs of the active differences.  The identified tangent space is
\begin{equation}
\label{eq:dm-aniso-tv-tangent}
   \mathcal T_\star=\ker D_Z.
\end{equation}
On the identified manifold, the TV term is locally affine:
\begin{equation}
\label{eq:dm-aniso-tv-local-affine}
   g(x)
     =\lambda\ip{\operatorname{sign}(D_Ap_\star)}{D_Ax}
     \quad\text{for }x\text{ near }p_\star
     \text{ with }D_Zx=0.
\end{equation}
Therefore the local proximal problem is a linearly perturbed projection onto \(\ker D_Z\), and
\begin{equation}
\label{eq:dm-aniso-tv-local-prox}
   P_{\gamma g}(w)
      =P_{\mathcal T_\star}
        \bigl(w-\gamma\lambda D_A^*\operatorname{sign}(D_Ap_\star)\bigr)
      \quad\text{locally},
\end{equation}
so that
\begin{equation}
\label{eq:dm-aniso-tv-Q}
   Q_\star=P_{\ker D_Z}.
\end{equation}

A two-pixel example makes the threshold geometry explicit.  Let
\(
 g(x)=\lambda|x_2-x_1|
\)
and put \(\delta=w_2-w_1\).  The proximal solution is fused when
\(
 |\delta|<2\gamma\lambda
\), in which case
\[
   P_{\gamma g}(w)
      =\left(\frac{w_1+w_2}{2},
              \frac{w_1+w_2}{2}\right),
   \qquad
   DP_{\gamma g}(w)
      =\frac12
       \begin{pmatrix}1&1\\[1mm]1&1\end{pmatrix}.
\]
The derivative is the projector onto the constant-vector subspace.  At
\(
 |\delta|=2\gamma\lambda
\), the fused and nonfused cells meet and the proximal map is not Fr\'echet differentiable.

\subsubsection{Isotropic or group total variation}

For block-valued finite differences, let
\begin{equation}
\label{eq:dm-iso-tv-g}
   g(x)=\lambda\sum_{r=1}^m\norm{(Dx)_r}_2.
\end{equation}
Define the active and zero blocks by
\begin{equation}
\label{eq:dm-iso-tv-active-sets}
   A:=\{r:(Dp_\star)_r\ne0\},
   \qquad
   Z:=\{r:(Dp_\star)_r=0\}.
\end{equation}
A dual certificate satisfies
\begin{equation}
\label{eq:dm-iso-tv-certificate}
   -\grad f(p_\star)=\lambda D^*q_\star,
   \qquad
   q_{\star,r}=\frac{(Dp_\star)_r}{\norm{(Dp_\star)_r}_2}
      \ (r\in A),
   \qquad
   \norm{q_{\star,r}}_2\le1
      \ (r\in Z).
\end{equation}
The block strict-complementarity condition is
\begin{equation}
\label{eq:dm-iso-tv-strict-complementarity}
   \norm{q_{\star,r}}_2<1
   \qquad\text{for every }r\in Z.
\end{equation}
The tangent space is again
\begin{equation}
\label{eq:dm-iso-tv-tangent}
   \mathcal T_\star=\ker D_Z.
\end{equation}
Unlike anisotropic TV, the restriction of \(g\) to the active manifold is smooth but not affine.  For \(r\in A\), set
\begin{equation}
\label{eq:dm-iso-tv-unit-blocks}
   u_r:=\frac{(Dp_\star)_r}{\norm{(Dp_\star)_r}_2},
   \qquad
   W_r:=\frac{1}{\norm{(Dp_\star)_r}_2}
          \bigl(I-u_ru_r^*\bigr),
\end{equation}
and let \(W_\star\) be the block-diagonal operator with blocks \(W_r\).  The Hessian of the active smooth TV expression is
\begin{equation}
\label{eq:dm-iso-tv-active-hessian}
   H_{\mathrm{TV},\star}
      =\lambda D_A^*W_\star D_A.
\end{equation}
If \(U\) is an isometry onto \(\mathcal T_\star\), Proposition~\ref{prop:dm-active-prox-derivative} gives
\begin{equation}
\label{eq:dm-iso-tv-Q}
   Q_\star
      =U\left[U^*\bigl(I+\gamma H_{\mathrm{TV},\star}\bigr)U\right]^{-1}U^*.
\end{equation}
Thus \(Q_\star\) is a self-adjoint positive contraction on the identified tangent space.  It reduces to the orthogonal projector only in the locally affine, polyhedral case.

\subsubsection{What fails without strict complementarity}

The nondegeneracy condition is not automatic.  For the \(\ell_1\) norm, it fails when a zero coordinate satisfies
\(
 |\grad_i f(p_\star)|=\lambda
\).  For anisotropic TV, it fails when the relevant dual certificate reaches
\(
 |q_{\star,j}|=1
\)
on a zero difference; for isotropic TV, the corresponding boundary is
\(
 \norm{q_{\star,r}}_2=1
\).  At such points, \(w_\star\) lies on the boundary of two or more active proximal cells.  A single Fr\'echet derivative \(Q_\star\) need not exist, so a fixed local linearization of the DRS map is not justified.

For \(\ell_1\) and anisotropic discrete TV, the proximal map remains piecewise affine and directionally differentiable.  One could replace \(Q_\star\) by a Bouligand derivative or a set of generalized Jacobians and obtain a piecewise-linear local DRS model.  That extension is possible but is outside the present analysis.  The current sufficient condition is therefore explicit: the active support or edge pattern must be nondegenerate in the sense of \eqref{eq:dm-ri-nondegeneracy}.

The TV statements above concern finite-dimensional discrete TV.  Continuum TV on \(BV\) or \(L^2\) involves moving jump sets and infinite-dimensional active geometry; Fr\'echet differentiability of its proximal map requires a separate analysis.  In the split TV--ADMM formulation used in Section~\ref{sec:numerics}, the nonsmooth variable is the finite-dimensional difference field itself, so the scalar or block threshold analysis above applies directly to that split proximal block.

\subsubsection{Local DRS linearization}

The preceding cases justify the differentiability assumption needed to expose the local forcing.  More generally, the next proposition only requires the derivative to exist.

\begin{proposition}[Local DRS linearization]
\label{prop:dm-local-drs-jacobian}
Assume \(P_{\gamma g}\) is Fr\'echet differentiable at \(w_\star\), with derivative
\(Q_\star:=DP_{\gamma g}(w_\star)\).  This assumption is satisfied under the hypotheses of Proposition~\ref{prop:dm-active-prox-derivative}, and in particular under the strict-complementarity conditions \eqref{eq:dm-l1-strict-complementarity}, \eqref{eq:dm-aniso-tv-strict-complementarity}, or \eqref{eq:dm-iso-tv-strict-complementarity}.  Then the exact DRS map
\[
    T=(1-\alpha)I+\alpha R_{\gamma g}R_{\gamma f}
\]
is Fr\'echet differentiable at \(z_\star\), and
\begin{equation}
\label{eq:dm-local-drs-jacobian}
    L_\star:=DT(z_\star)
    =(1-\alpha)I
      +\alpha(2Q_\star-I)(2G_\star^{-1}-I).
\end{equation}
Consequently, for the exact DRS iteration,
\begin{equation}
\label{eq:dm-local-outer-increment}
    \Delta z_k
    =(L_\star-I)(z_{k-1}-z_\star)
      +o(\norm{z_{k-1}-z_\star}).
\end{equation}
The same expansion holds for an inexact CR-DRS iteration whenever its outer perturbation is \(o(\norm{z_{k-1}-z_\star})\).
\end{proposition}

\begin{proof}
The optimality equation for \(P_{\gamma f}\) is
\(\gamma\grad f(p)+p=z\).  The implicit-function theorem gives
\(DP_{\gamma f}(z_\star)=G_\star^{-1}\) and hence
\(DR_{\gamma f}(z_\star)=2G_\star^{-1}-I\).  Also
\(DR_{\gamma g}(w_\star)=2Q_\star-I\).  The chain rule gives
\eqref{eq:dm-local-drs-jacobian}; Taylor expansion of \(T\) gives
\eqref{eq:dm-local-outer-increment}.  An \(o(\norm{z_{k-1}-z_\star})\) perturbation does not change the first-order term.
\end{proof}

Because a proximal map is firmly nonexpansive, its derivative satisfies
\begin{equation}
\label{eq:dm-Q-firm}
    \norm{Q_\star h}^2\le\ip{Q_\star h}{h}
    \qquad(h\in\Hcal).
\end{equation}
Thus \(Q_\star\) is a contraction.  In the \(\ell_1\) and anisotropic-TV cases it is an orthogonal projector; for isotropic TV it is the positive contraction \eqref{eq:dm-iso-tv-Q}.  Formula~\eqref{eq:dm-local-drs-jacobian} now gives the methodological point: after active identification, the new right-hand sides are not arbitrary external vectors.  They are generated, to first order, by repeated action of one fixed operator \(L_\star\).

The next statement turns this observation into the active-subspace condition of Proposition~\ref{prop:dm-active-subspace}.

\begin{proposition}[DRS-generated active forcing]
\label{prop:dm-active-forcing}
Let \(\mathcal U\subseteq\Hcal\) be an \(L_\star\)-invariant subspace and suppose
\begin{align}
    \frac{\dist(z_k-z_\star,\mathcal U)}{\norm{z_k-z_\star}}&\to0,
    \label{eq:dm-outer-active-subspace}\\
    \norm{(I-L_\star)(z_k-z_\star)}&\ge
       c_z\norm{z_k-z_\star}
       \quad\text{eventually}
    \label{eq:dm-increment-nondegenerate}
\end{align}
for some \(c_z>0\).  Suppose also that
\begin{equation}
\label{eq:dm-anchor-smaller-than-motion}
    \norm{a_{k-1}-p_{k-1}}=o(\norm{\Delta z_k}).
\end{equation}
Define
\begin{equation}
\label{eq:dm-active-forcing-subspace}
    \mathcal V:=G_\star^{-1}(I-L_\star)\mathcal U.
\end{equation}
Then
\begin{equation}
\label{eq:dm-active-direction-confinement}
    \frac{\dist(d_k,\mathcal V)}{\norm{d_k}}\longrightarrow0.
\end{equation}
Consequently, if \(B_k\to G_\star\) on \(\mathcal V\) in the sense of
\eqref{eq:dm-subspace-model}, then the transported Dennis--Mor\'e condition holds.
\end{proposition}

\begin{proof}
Let \(\zeta_{k-1}=z_{k-1}-z_\star\).  Equations~\eqref{eq:dm-local-outer-increment} and
\eqref{eq:dm-outer-active-subspace}, together with invariance of \(\mathcal U\), give
\[
    \dist(\Delta z_k,(L_\star-I)\mathcal U)
       =o(\norm{\zeta_{k-1}}).
\]
Condition~\eqref{eq:dm-increment-nondegenerate} implies
\(\norm{\Delta z_k}\ge(c_z+o(1))\norm{\zeta_{k-1}}\), so the preceding distance is
\(o(\norm{\Delta z_k})\).  Since \(G_{k-1}^{-1}\to G_\star^{-1}\),
\eqref{eq:dm-direction-decomposition} and
\eqref{eq:dm-anchor-smaller-than-motion} yield
\[
    \dist(d_k,\mathcal V)=o(\norm{\Delta z_k}).
\]
The same relations and positive definiteness of \(G_\star\) give
\(\norm{d_k}\ge c_d\norm{\Delta z_k}\) for some \(c_d>0\) and all large \(k\).
This proves \eqref{eq:dm-active-direction-confinement}.  Proposition~\ref{prop:dm-active-subspace} gives the final claim.
\end{proof}

For \(g(x)=\lambda\norm{x}_1\), the active subspace in this proposition is generated by a DRS linearization whose nonsmooth factor is the support projector \eqref{eq:dm-l1-Q}.  For anisotropic TV it is generated by the edge-tangent projector \eqref{eq:dm-aniso-tv-Q}; for isotropic TV, by the positive contraction \eqref{eq:dm-iso-tv-Q}.  Hence strict complementarity supplies a concrete chain:
\[
   \begin{gathered}
   \text{stable support or edge pattern}
   \Longrightarrow Q_\star
   \Longrightarrow L_\star\\
   \Longrightarrow \mathcal V
   \Longrightarrow
     \begin{gathered}
       \text{a subspace on which}\\
       B_k\text{ must learn }G_\star
     \end{gathered}.
   \end{gathered}
\]
This does not itself prove the transported Dennis--Mor\'e condition.  It shows, however, that the outer forcing is structured by the identified nonsmooth geometry and that the quasi-Newton model need only become accurate on the resulting DRS-generated subspace.

If \(f\) is quadratic and \(P_{\gamma g}\) is affine on the identified region, then \(T\) is affine there.  The outer errors and increments lie exactly in the finite Krylov space
\begin{equation}
\label{eq:dm-krylov-space}
    \mathcal K
    :=\operatorname{span}\{(z_{k_0}-z_\star),
        L_\star(z_{k_0}-z_\star),
        L_\star^2(z_{k_0}-z_\star),\ldots\}.
\end{equation}
In finite dimensions this chain stabilizes after at most \(\dim\Hcal\) steps, and often much earlier when only a few slow DRS modes remain.  Thus the counterexample above does not imply perpetual failure: it identifies the need to learn the stabilized subspace
\(
 G_\star^{-1}(I-L_\star)\mathcal K
\).

\subsection{Learning the active subspace}

The active-subspace requirement can be verified exactly for the rank-one update in a locally quadratic model.

\begin{proposition}[Finite SR1 learning on a quadratic active subspace]
\label{prop:dm-sr1-active-learning}
Let $G_\star$ be constant and suppose every accepted exact secant satisfies
$y_i=G_\star s_i$.  Let $B_{i+1}$ be obtained from $B_i$ by the direct SR1 update \eqref{eq:direct-sr1}, and assume each displayed update is well defined.  Set
$E_i:=B_i-G_\star$.  Then
\begin{equation}
\label{eq:dm-sr1-error-update}
    E_{i+1}
    =E_i-\frac{E_is_i s_i^\top E_i}{s_i^\top E_i s_i},
\end{equation}
and
\begin{equation}
\label{eq:dm-sr1-kernel-inheritance}
    E_{i+1}s_i=0,
    \qquad
    \ker(E_i)\subseteq\ker(E_{i+1}).
\end{equation}
Consequently, if a finite set of accepted directions $s_{i_1},\ldots,s_{i_r}$ spans a subspace $\mathcal V$, then after those updates
\begin{equation}
\label{eq:dm-sr1-exact-on-active}
    (B-G_\star)v=0
    \qquad\text{for every }v\in\mathcal V.
\end{equation}
In particular, transported Dennis--Mor\'e holds for every later transported direction contained in $\mathcal V$.
\end{proposition}

\begin{proof}
With $y_i=G_\star s_i$, the SR1 correction vector is
$y_i-B_is_i=-E_is_i$.  Substitution into \eqref{eq:direct-sr1} gives
\eqref{eq:dm-sr1-error-update}.  Multiplying by $s_i$ gives
$E_{i+1}s_i=0$.  If $E_iv=0$, then the correction term in
\eqref{eq:dm-sr1-error-update} also vanishes on $v$, proving kernel inheritance.
Each accepted independent direction therefore enlarges the kernel of the Hessian error, and spanning of $\mathcal V$ gives \eqref{eq:dm-sr1-exact-on-active}.
\end{proof}

This proposition gives a complete deterministic route in the quadratic, locally affine case.  Regularity of $g$ confines the DRS forcing to a fixed active subspace.  Accepted SR1 secants then recover the true curvature on that subspace.  Extension to a varying Hessian requires additional stability and denominator bounds.

BFGS does not have the same kernel-inheritance identity for arbitrary secant directions.  It preserves all earlier secant equations under additional structure, for example suitable conjugacy, and classical full-memory theory can yield full-model or Dennis--Mor\'e convergence under its usual local assumptions.  For persistent BFGS or fixed-memory L-BFGS driven by general DRS directions, proving \eqref{eq:dm-subspace-model} remains a separate task.  The local structure of $g$ makes that task plausible by reducing it to a fixed, often low-dimensional subspace; it does not make it automatic.

The theoretical status is therefore precise.  Exact pairing removes the inverse-consistency condition.  A valid DRS step can still request an unlearned direction.  Transported Dennis--Mor\'e follows from full-model convergence or from convergence on the DRS-generated active subspace.  For quadratic curvature, paired SR1 learns any finitely excited active subspace exactly.  Establishing an analogous active-subspace convergence theorem for persistent BFGS/L-BFGS and general nonlinear $f$ is the main remaining local question.

\section{Practical implementation and safeguards}
\label{sec:heuristics}

The certified algorithm is the theorem-bearing form.  The following choices describe the faster implementation used in the numerical tests.  None of them introduces a line search inside the transported predictor.

\paragraph{Predictor depth}
With a frozen paired model, Proposition~\ref{prop:frozen-cancellation-new} gives
$\eta_k^{\rm eff}=1-\prod_j(1-\eta_{k,j})$.  Under exact pairing and $\nu_{\rm qn}=1$, the rule \eqref{eq:predictor-eta} reduces to $\eta_{k,j}=\bar\eta_k$.  Thus a fixed conservative cap such as $\bar\eta_k\in[0.3,0.4]$ makes $K=4$ or $K=8$ move much closer to the same model minimizer, while increasing $K$ further gives little extra displacement.

\paragraph{Model check and damping schedule}
After the first true current gradient, record
\begin{equation}
\label{eq:model-check}
    \epsilon_k^{\rm model}
    :=\frac{\norm{r_k(\widehat x_k)-\widetilde r_{k,K}}}
            {\max\{1,\norm{r_k(\widehat x_k)}\}}.
\end{equation}
A fully explicit, no-line-search schedule is obtained by fixing
$\eta_{\rm safe}\in(0,1)$ and $\tau_\eta>0$ and setting
\begin{equation}
\label{eq:damping-trust-switch}
  \bar\eta_{k+1}:=
  \begin{cases}
    1, & \epsilon_k^{\rm model}\le\tau_\eta
         \ \text{and the endpoint tests were accepted},\\
    \eta_{\rm safe}, & \text{otherwise}.
  \end{cases}
\end{equation}
The initial cap is $\bar\eta_0:=\eta_{\rm safe}$.  Rule \eqref{eq:damping-trust-switch} uses only information already obtained at outer iteration $k$; it does not evaluate any new trial point at iteration $k+1$.  If the model remains accurate, it also permits the condition $\eta_k^{\rm eff}\to1$ in Theorem~\ref{thm:transported-dm-focused}.

\paragraph{Direct/inverse consistency}
For a nonzero probe vector $v$, define
\begin{equation}
\label{eq:BH-consistency}
  \epsilon_k^{BH}(v)
  :=\frac{\norm{H_k(B_kv)-v}}{\norm{v}}.
\end{equation}
The reference probe is $v=\bar r_k(a_{k-1})$ when this vector is nonzero.  A threshold $\tau_{BH}>0$ is fixed in advance; if $\epsilon_k^{BH}(v)>\tau_{BH}$, the model is rebuilt from $\mathcal M_k$ or restarted from \eqref{eq:initial-qn-scaling}.

\paragraph{Displacement cap}
Choose $\chi_\Delta>0$.  For $k\ge1$, set
\begin{equation}
\label{eq:predictor-radius}
  \Delta_k:=\chi_\Delta
  \max\left\{\norm{z_k-z_{k-1}},
              \gamma\norm{\bar r_k(a_{k-1})}\right\},
\end{equation}
and set $\Delta_0:=\chi_\Delta\gamma\norm{\bar r_0(a_{-1})}$.  Let $\eta_{k,j}^{\rm raw}$ denote the value in \eqref{eq:predictor-eta}.  The capped value is
\begin{equation}
\label{eq:predictor-radius-cap}
  \eta_{k,j}:=
  \sup\left\{\eta\in[0,\eta_{k,j}^{\rm raw}]:
  \norm{\widetilde x_{k,j}+\eta d_{k,j}-a_{k-1}}\le\Delta_k\right\}.
\end{equation}
This is a one-dimensional quadratic calculation.  It needs no value or gradient evaluation.

\paragraph{Other safeguards}
The certified reference algorithm checks finite directions, positive model descent through \eqref{eq:predictor-eta}, the endpoint tests \eqref{eq:descent-direction-certificate}--\eqref{eq:residual-certificate}, the exact-pair curvature threshold \eqref{eq:pair-curvature-threshold}, and the consistency test \eqref{eq:BH-consistency}.  SR1 updates are skipped when \eqref{eq:sr1-skip} or \eqref{eq:sr1-spectral-bounds} fails.  Pseudo pairs of the form $y=B_ks$ are never stored.

\paragraph{Cost}
For dimension $n$, finite memory $m_q$, and depth $K$, an L-BFGS inverse action costs $O(m_qn)$.  A simple replay of the paired direct updates costs $O(m_q^2n)$, and the frozen predictor costs $O(Km_q^2n)$.  Storage is $O(m_qn)$.  The intended regime is one in which a new $\grad f$ costs much more than this linear algebra.

\paragraph{PRS and ADMM}
PRS is obtained by taking $\alpha=1$ whenever the contraction assumptions allow it.  For the ADMM form below, let $A:\Hcal\to\Kcal$ be a bounded linear map between real Hilbert spaces, let $A^*$ be its adjoint, let $c_k\in\Kcal$ be the current affine center, and let $\rho_{\!A}>0$ be the fixed quadratic-penalty parameter.  The expensive block often has the form
\begin{equation}
\label{eq:admm-affine-family-focused}
    \psi_k(x)=f(x)+\frac{\rho_{\!A}}{2}\norm{Ax-c_k}^2,
    \qquad
    \grad^2\psi_k(x)=\grad^2f(x)+\rho_{\!A} A^*A.
\end{equation}
ADMM is DRS on the Fenchel dual.  Again, the changing outer variables enter only through an affine term.  Residual transport, BFGS and SR1 reuse, and the predictor apply with $I$ replaced by the known quadratic curvature $\rho_{\!A}A^*A$.

\section{Numerical results}
\label{sec:numerics}

The experiments test the fixed-budget predictor, not the fallback branch of Algorithm~\ref{alg:certified-crdrs}.  Their purpose is to measure progress per expensive gradient.  They do not by themselves prove that the fixed-budget runs satisfy the certificates of Section~\ref{sec:global}.  Let $F_\star$ denote a high-accuracy reference objective value.  Every reported relative objective gap is
\begin{equation}
\label{eq:numerical-relative-gap}
  G(x):=\frac{\max\{F(x)-F_\star,0\}}
              {\max\{1,|F_\star|\}}.
\end{equation}
For a reference solution $x_\star^{\rm ref}$, the reported relative reconstruction error is
\begin{equation}
\label{eq:numerical-relative-error}
  E_{\rm rel}(x):=
  \frac{\norm{x-x_\star^{\rm ref}}}
       {\norm{x_\star^{\rm ref}}},
  \qquad x_\star^{\rm ref}\ne0.
\end{equation}
A ``paired improvement'' is computed for each seed first and then summarized by its median.

\subsection{Direct DRS with \texorpdfstring{$\ell_1$}{l1} regularization}

We use three convex $n=128$ problem families with $g(x)=\lambda\norm{x}_1$ and $\lambda>0$: an ill-conditioned quadratic inverse problem, correlated logistic regression, and a robust pseudo-Huber inverse problem.  Each result is the median over eight seeds.  Unless stated otherwise, every method receives 200 true gradients: 100 outer DRS iterations and two gradients per smooth proximal call.  The L-BFGS memory is ten.

\begin{table}[t]
\centering
\caption{Median final relative objective gap at 200 true gradients.  ``Reset'' warm-starts the state but resets L-BFGS; ``persistent'' keeps curvature but has no transported displacement.}
\label{tab:focused-main}
\scriptsize
\begin{tabular}{lrrrrr}
\toprule
Problem & Reset & Persistent & $K=1$ & $K=8$ & $K=16$\\
\midrule
Quadratic & $1.97\!\times\!10^{-2}$ & $5.59\!\times\!10^{-3}$ & $6.70\!\times\!10^{-4}$ & $2.82\!\times\!10^{-5}$ & $2.66\!\times\!10^{-5}$\\
Logistic & $1.41\!\times\!10^{-3}$ & $2.73\!\times\!10^{-4}$ & $4.10\!\times\!10^{-5}$ & $6.33\!\times\!10^{-6}$ & $6.16\!\times\!10^{-6}$\\
Pseudo-Huber & $4.07\!\times\!10^{-2}$ & $2.34\!\times\!10^{-2}$ & $1.49\!\times\!10^{-2}$ & $9.56\!\times\!10^{-3}$ & $9.48\!\times\!10^{-3}$\\
\bottomrule
\end{tabular}
\end{table}

At equal cost, $K=8$ improves the paired final gap over $K=1$ by $96.96\%$, $88.04\%$, and $37.54\%$ on the quadratic, logistic, and pseudo-Huber families.  It wins all 24 seedwise comparisons.  The small change from $K=8$ to $K=16$ is consistent with the effective-damping saturation in Proposition~\ref{prop:frozen-cancellation-new}.

\begin{figure}[t]
\centering
\begin{subfigure}[t]{0.31\textwidth}
\includegraphics[width=\linewidth]{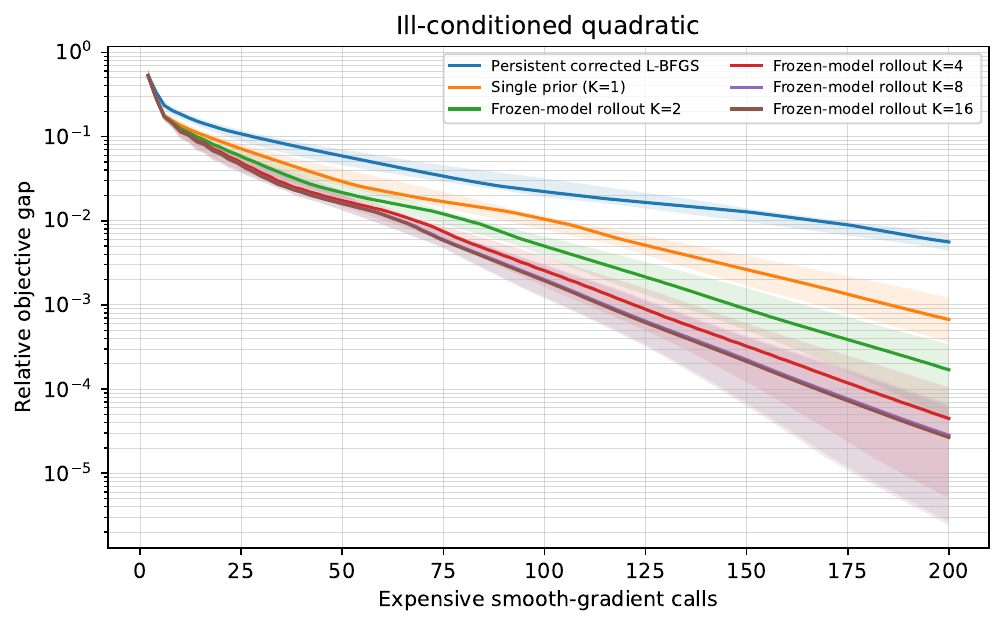}
\caption{Quadratic}
\end{subfigure}\hfill
\begin{subfigure}[t]{0.31\textwidth}
\includegraphics[width=\linewidth]{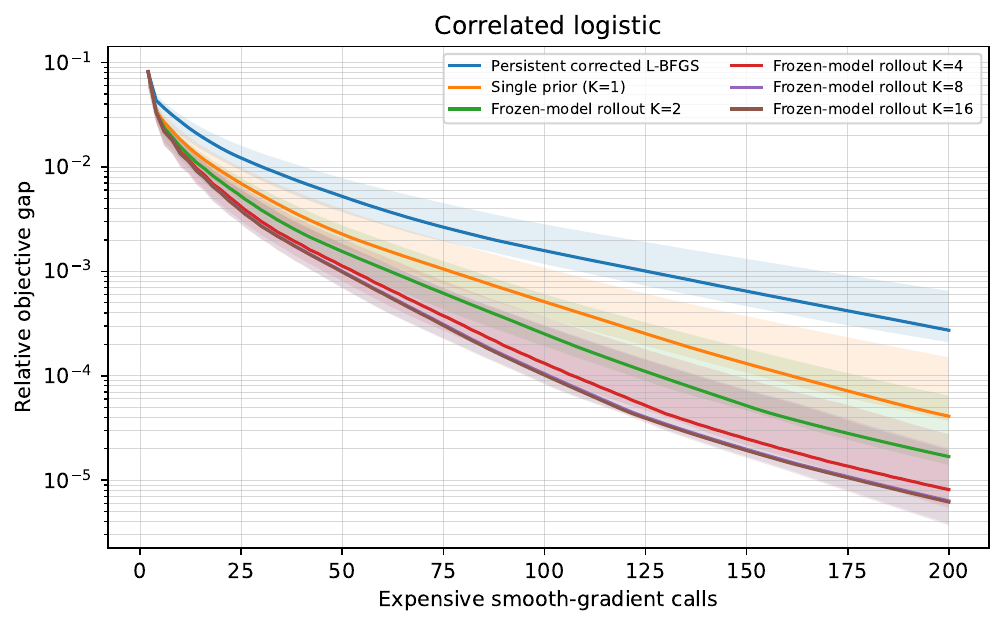}
\caption{Logistic}
\end{subfigure}\hfill
\begin{subfigure}[t]{0.31\textwidth}
\includegraphics[width=\linewidth]{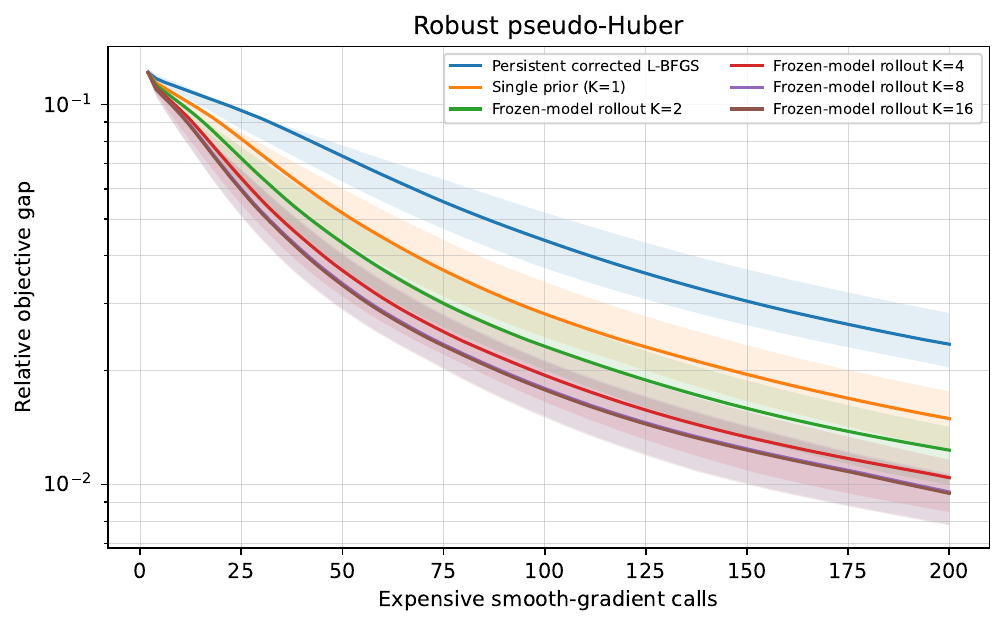}
\caption{Pseudo-Huber}
\end{subfigure}
\caption{Median relative objective gap versus true smooth-gradient calls.  Shaded regions are interquartile ranges over eight seeds.}
\label{fig:focused-objective-curves}
\end{figure}

\subsection{One-gradient outer iterations}

The most aggressive runs use one true gradient per outer iteration.  Table~\ref{tab:focused-one-gradient} compares these runs with the earlier $K=1$ method and with much larger non-predictive inner budgets.

\begin{table}[t]
\centering
\caption{Median final gap after 100 outer iterations.  Parentheses give total true-gradient calls.}
\label{tab:focused-one-gradient}
\scriptsize
\begin{tabular}{@{}lrrrr@{}}
\toprule
Problem & $K=8$ (100) & $K=16$ (100) & $K=1$ (200) & Persistent (1000)\\
\midrule
Quadratic & $2.06\!\times\!10^{-5}$ & $1.83\!\times\!10^{-5}$ & $6.70\!\times\!10^{-4}$ & $2.56\!\times\!10^{-5}$\\
Logistic & $5.22\!\times\!10^{-6}$ & $4.75\!\times\!10^{-6}$ & $4.10\!\times\!10^{-5}$ & $6.25\!\times\!10^{-6}$\\
Pseudo-Huber & $9.66\!\times\!10^{-3}$ & $9.38\!\times\!10^{-3}$ & $1.49\!\times\!10^{-2}$ & $9.55\!\times\!10^{-3}$\\
\bottomrule
\end{tabular}
\end{table}

With 100 gradients, $K=8$ improves the paired gap over the $K=1$, 200-gradient method by $98.03\%$, $91.10\%$, and $36.61\%$.  These are finite-budget empirical results.  A production implementation should also monitor the smooth-prox residual; a small approximate DRS residual alone can be misleading when the smooth proximal problem is severely under-solved.

\subsection{Comparison with FISTA}

For directly proximal $\ell_1$ regularization, FISTA is a natural first-order baseline \cite{BeckTeboulle2009}.  All methods in Table~\ref{tab:focused-fista} use 200 gradients.

\begin{table}[t]
\centering
\caption{Median final relative objective gap at 200 gradients.  The last column is the median paired improvement of $K=8$ over FISTA.}
\label{tab:focused-fista}
\scriptsize
\begin{tabular}{@{}lrrrr@{}}
\toprule
Problem & FISTA & CR $K=1$ & CR $K=8$ & Paired gain\\
\midrule
Quadratic & $1.664\!\times\!10^{-3}$ & $6.701\!\times\!10^{-4}$ & $2.824\!\times\!10^{-5}$ & $98.68\%$\\
Logistic & $8.877\!\times\!10^{-5}$ & $4.100\!\times\!10^{-5}$ & $6.327\!\times\!10^{-6}$ & $93.53\%$\\
Pseudo-Huber & $1.779\!\times\!10^{-2}$ & $1.491\!\times\!10^{-2}$ & $9.561\!\times\!10^{-3}$ & $47.26\%$\\
\bottomrule
\end{tabular}
\end{table}

On the quadratic family, the median number of calls needed to reach a relative gap of $10^{-3}$ falls from 190 for FISTA to 116 for $K=8$.  On logistic regression it falls from 94 to 51.  These comparisons concern the tested structured problems and budgets; they do not imply uniform dominance over accelerated proximal-gradient methods.

\subsection{TV deconvolution through ADMM}

We also test
\begin{equation}
\label{eq:tv-problem-focused}
    \min_x \frac12\norm{Ax-b_{\rm obs}}^2+\lambda\norm{Dx}_1,
\end{equation}
where $A$ is the blur operator, $b_{\rm obs}$ is the observed data, $D$ is the discrete first-difference operator, and $\lambda>0$ is the TV weight.  The $x$-subproblem has the affine structure \eqref{eq:admm-affine-family-focused}.  All variants use 120 ADMM iterations and 240 true gradients.  Here $K=0$ keeps the persistent Hessian but applies no transported state displacement.

\begin{table}[t]
\centering
\caption{Median TV-deconvolution results.  Gap reduction is relative to the curvature-only $K=0$ baseline.}
\label{tab:focused-tv}
\scriptsize
\begin{tabular}{@{}llrrr@{}}
\toprule
Case & $K$ & Relative gap & Gap reduction & Rel. reconstruction error\\
\midrule
1D & 0 & $3.910\!\times\!10^{-3}$ & $0.00\%$ & $5.573\!\times\!10^{-2}$\\
1D & 1 & $2.731\!\times\!10^{-3}$ & $30.15\%$ & $4.967\!\times\!10^{-2}$\\
1D & 8 & $2.572\!\times\!10^{-3}$ & $34.21\%$ & $4.824\!\times\!10^{-2}$\\
1D & 16 & $2.569\!\times\!10^{-3}$ & $34.31\%$ & $4.822\!\times\!10^{-2}$\\
\midrule
2D & 0 & $2.564\!\times\!10^{-2}$ & $0.00\%$ & $2.293\!\times\!10^{-1}$\\
2D & 1 & $1.373\!\times\!10^{-2}$ & $46.46\%$ & $2.154\!\times\!10^{-1}$\\
2D & 8 & $1.010\!\times\!10^{-2}$ & $60.60\%$ & $2.087\!\times\!10^{-1}$\\
2D & 16 & $1.007\!\times\!10^{-2}$ & $60.72\%$ & $2.085\!\times\!10^{-1}$\\
\bottomrule
\end{tabular}
\end{table}

\begin{figure}[t]
\centering
\begin{subfigure}[t]{0.48\textwidth}
\includegraphics[width=\linewidth]{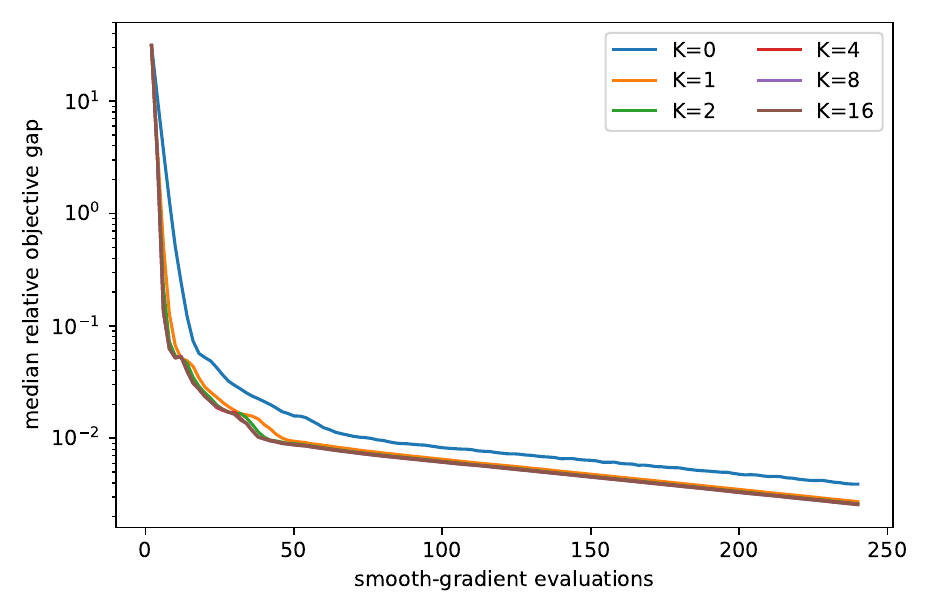}
\caption{One-dimensional deconvolution}
\end{subfigure}\hfill
\begin{subfigure}[t]{0.48\textwidth}
\includegraphics[width=\linewidth]{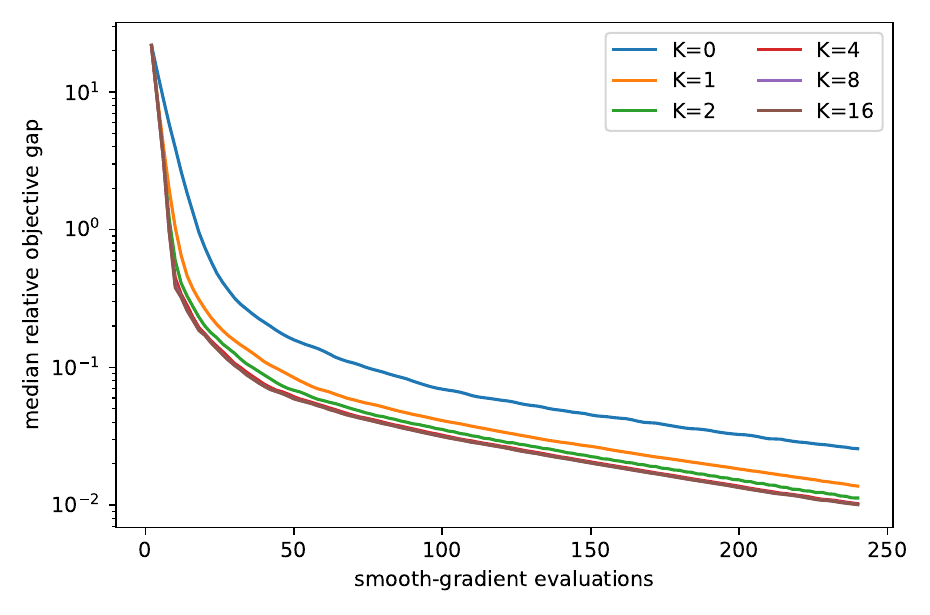}
\caption{Two-dimensional deconvolution}
\end{subfigure}
\caption{Median objective gap versus true gradient evaluations for the ADMM/dual-DR experiment.}
\label{fig:focused-tv-curves}
\end{figure}

\subsection{Scope of the evidence}

The benchmarks are convex and small enough that their gradients are cheap.  They establish reductions in gradient calls, not application-scale wall-clock speedups.  The experiments use L-BFGS; the SR1 realization is included in the formal method but is not tested here.  The fixed-budget runs also do not enforce the fallback certificate in Algorithm~\ref{alg:certified-crdrs}.  The next computational step is therefore to combine the certified acceptance rule with an expensive forward/adjoint model and report end-to-end CPU/GPU timings.

\section{Conclusion}
\label{sec:conclusion-focused}

The smooth proximal problems generated by DRS are not independent.  Their centers change, but their nonlinear curvature is shared.  CR-DRS exploits this by transporting exact residuals, retaining paired BFGS or SR1 curvature, and moving the old proximal state before the next expensive gradient.

The certified algorithm gives a clean division of labor.  The quasi-Newton candidate provides speed.  The descent and residual tests provide a computable certificate.  A fixed safe fallback provides global convergence.  Under the small-gain condition, the outer DRS error and inner proximal-tracking error converge linearly with a constant gradient budget.  Near the solution, exact direct/inverse pairing and the transported Dennis--Mor\'e condition make the candidate superlinear relative to an ordinary warm start; the certificate then accepts it with one gradient.

The transported Dennis--Mor\'e condition is not a formal consequence of DRS.  A valid outer step can request curvature in a direction absent from the current memory.  The local DRS linearization nevertheless shows that, for regular nonsmooth terms, these right-hand-side corrections are generated by a fixed tangent-space dynamics rather than by arbitrary forcing.  Full-model convergence or convergence on that active subspace is sufficient.  In the quadratic case, accepted SR1 directions learn any finitely excited active subspace exactly.  The main remaining local question is an analogous active-subspace convergence result for persistent BFGS/L-BFGS with nonlinear curvature.

The experiments support the computational claim: when gradients are the expensive resource, state transport adds substantial value beyond curvature persistence alone.  The next computational step is application-scale timing with an expensive forward/adjoint model and the certified acceptance rule.

\section*{Acknowledgments}
The author thanks Dr.~Stewart Levin for helpful discussions and suggestions.

\appendix
\section{Notation and basic proximal facts}
\label{app:proximal-facts}

Throughout, $\Hcal$ is a real Hilbert space with inner product $\ip{\cdot}{\cdot}$ and norm $\norm{\cdot}$.  In finite dimensions the same notation is used for Euclidean spaces.  We write $\Gamma_0(\Hcal)$ for the class of proper, lower semicontinuous, convex functions from $\Hcal$ to $\R\cup\{+\infty\}$.  For $f\in\Gamma_0(\Hcal)$, its subdifferential at $x\in\Hcal$ is
\begin{equation}
\label{eq:subdifferential-definition}
  \subd f(x)
  :=\left\{a\in\Hcal:\
  f(y)\ge f(x)+\ip{a}{y-x}\quad\text{for every }y\in\Hcal\right\}.
\end{equation}
If $a\in\subd f(x)$ and $b\in\subd f(y)$, adding the two subgradient inequalities gives
\begin{equation}
\label{eq:subgradient-monotonicity}
  \ip{x-y}{a-b}\ge0.
\end{equation}
Thus $\subd f$ is monotone; in fact it is maximally monotone.  These facts, together with the resolvent and proximity-operator characterizations used below, are developed in \cite[Thm.~20.25, Prop.~23.2, Ex.~23.3, and Prop.~24.1]{BauschkeCombettes2017}.

For $\gamma>0$ and $u\in\Hcal$, define the proximal and reflected proximal operators by
\begin{align}
\label{eq:proximal-definition}
  P_{\gamma f}u
  :=\prox_{\gamma f}u
  &:=\argmin_{x\in\Hcal}
  \left\{f(x)+\frac{1}{2\gamma}\norm{x-u}^2\right\}
  =(\Id+\gamma\subd f)^{-1}u,\\
\label{eq:reflected-proximal-definition}
  R_{\gamma f}u&:=2P_{\gamma f}u-u.
\end{align}
The objective in \eqref{eq:proximal-definition} is strongly convex and has a unique minimizer.  Its first-order optimality condition gives the fundamental equivalence
\begin{equation}
\label{eq:proximal-optimality}
  p=P_{\gamma f}u
  \quad\Longleftrightarrow\quad
  \frac{u-p}{\gamma}\in\subd f(p)
  \quad\Longleftrightarrow\quad
  u\in p+\gamma\subd f(p).
\end{equation}
The same notation and identities will be used for every function in a splitting pair, in particular for $P_{\gamma g}$ and $R_{\gamma g}$.

\begin{lemma}[firm nonexpansiveness of proximal maps]
\label{lem:prox-firmly-nonexpansive}
Let $f\in\Gamma_0(\Hcal)$ and $\gamma>0$.  For every $u,v\in\Hcal$, set $p=P_{\gamma f}u$ and $q=P_{\gamma f}v$.  Then
\begin{equation}
\label{eq:firm}
  \norm{p-q}^2\le \ip{p-q}{u-v},
\end{equation}
and, equivalently,
\begin{equation}
\label{eq:firm-pythagorean}
  \norm{p-q}^2
  +\norm{(u-p)-(v-q)}^2
  \le \norm{u-v}^2.
\end{equation}
Consequently, both $P_{\gamma f}$ and $\Id-P_{\gamma f}$ are firmly nonexpansive, and
\begin{equation}
\label{eq:prox-nonexpansive}
  \norm{P_{\gamma f}u-P_{\gamma f}v}\le\norm{u-v}.
\end{equation}
\end{lemma}

\begin{proof}
By \eqref{eq:proximal-optimality},
\[
  \frac{u-p}{\gamma}\in\subd f(p),
  \qquad
  \frac{v-q}{\gamma}\in\subd f(q).
\]
Applying the monotonicity inequality \eqref{eq:subgradient-monotonicity} to these two graph points yields
\[
  0\le
  \ip{p-q}{\frac{u-p}{\gamma}-\frac{v-q}{\gamma}}
  =\frac{1}{\gamma}
  \ip{p-q}{(u-v)-(p-q)}.
\]
After multiplying by $\gamma$ and rearranging, this is precisely \eqref{eq:firm}.

For the stronger two-term form, write
\[
  u-v=(p-q)+\bigl[(u-p)-(v-q)\bigr].
\]
Expanding the squared norm and using the nonnegativity of the cross term just proved gives
\begin{align*}
  \norm{u-v}^2
  &=\norm{p-q}^2
    +\norm{(u-p)-(v-q)}^2\\
  &\quad
    +2\ip{p-q}{(u-p)-(v-q)}\\
  &\ge
    \norm{p-q}^2
    +\norm{(u-p)-(v-q)}^2,
\end{align*}
which is \eqref{eq:firm-pythagorean}.  This inequality is the defining firm-nonexpansiveness estimate for $P_{\gamma f}$; because its two terms are symmetric, it is also the corresponding estimate for $\Id-P_{\gamma f}$.  Finally, \eqref{eq:prox-nonexpansive} follows immediately by dropping the second nonnegative term in \eqref{eq:firm-pythagorean}.  Equivalently, it follows from \eqref{eq:firm} and Cauchy--Schwarz.  See also \cite[Prop.~4.4, Prop.~23.8, and Cor.~23.11]{BauschkeCombettes2017}.
\end{proof}

The firm estimate is stronger than ordinary nonexpansiveness: it controls both the displacement of the proximal points and the displacement of the complementary residuals $u-P_{\gamma f}u$ and $v-P_{\gamma f}v$.  The same estimate also gives nonexpansiveness of the reflected proximal map.

\begin{lemma}[reflected proximal maps are nonexpansive]
\label{lem:reflected-nonexpansive}
Let $f\in\Gamma_0(\Hcal)$ and $\gamma>0$.  Then
\begin{equation}
\label{eq:reflected-nonexpansive}
  \norm{R_{\gamma f}u-R_{\gamma f}v}\le\norm{u-v}
  \qquad\text{for all }u,v\in\Hcal.
\end{equation}
\end{lemma}

\begin{proof}
Let $p=P_{\gamma f}u$, $q=P_{\gamma f}v$, and set
\[
  d:=p-q,
  \qquad
  e:=(u-p)-(v-q).
\]
Then $u-v=d+e$ and, by \eqref{eq:firm}, $\ip{d}{e}\ge0$.  Moreover,
\[
  R_{\gamma f}u-R_{\gamma f}v
  =(p-(u-p))-(q-(v-q))=d-e.
\]
Therefore
\begin{align*}
  \norm{R_{\gamma f}u-R_{\gamma f}v}^2
  &=\norm{d-e}^2\\
  &=\norm d^2+\norm e^2-2\ip d e\\
  &\le \norm d^2+\norm e^2+2\ip d e
   =\norm{d+e}^2
   =\norm{u-v}^2.
\end{align*}
Taking square roots proves \eqref{eq:reflected-nonexpansive}.  This is also the equivalence between firm nonexpansiveness of $P_{\gamma f}$ and nonexpansiveness of $2P_{\gamma f}-\Id$ in \cite[Prop.~4.4]{BauschkeCombettes2017}.
\end{proof}

\section{Douglas--Rachford splitting}
\label{app:exact-drs}

We consider \eqref{eq:main-problem} in a real Hilbert space $\Hcal$, where $f,g\in\Gamma_0(\Hcal)$.  For an operator $S$, $\Fix(S):=\{x:Sx=x\}$; for a set-valued operator $A$, $\operatorname{zer}A:=\{x:0\in Ax\}$ and $\operatorname{gra}A:=\{(x,u):u\in Ax\}$.  The notation $u_n\rightharpoonup u$ means weak convergence.  In \eqref{eq:DR-sum-rule-qualification}, $\operatorname{sri}$ denotes the strong relative interior.  Douglas--Rachford splitting originates with the alternating-direction construction of \cite{DouglasRachford1956} and was extended to maximal monotone operators by \cite{LionsMercier1979}.  Modern resolvent and minimization formulations are given, for example, in \cite[Secs.~26.3 and 28.3]{BauschkeCombettes2017} and \cite{CombettesPesquet2011}.

For a fixed proximal parameter $\gamma>0$, define
\begin{equation}
\label{eq:DR-reflection-composition}
  N:=R_{\gamma g}R_{\gamma f},
\end{equation}
and, for a relaxation parameter $\alpha>0$, define the relaxed Douglas--Rachford operator
\begin{equation}
\label{eq:DR-operator}
  T:=(1-\alpha)\Id+\alpha N.
\end{equation}
The fixed-point iteration $z_{k+1}=Tz_k$ can be written as
\begin{align}
\label{eq:DR-proxf-first-x}
  x_k&=P_{\gamma f}z_k,\tag{PROXF}\\
\label{eq:DR-proxf-first-y}
  y_k&=P_{\gamma g}(2x_k-z_k),\tag{PROXG}\\
\label{eq:DR-proxf-first-z}
  z_{k+1}&=z_k+2\alpha(y_k-x_k).
\end{align}
Indeed, $R_{\gamma f}z_k=2x_k-z_k$ and
$R_{\gamma g}R_{\gamma f}z_k=2y_k-(2x_k-z_k)$, so
\[
  (1-\alpha)z_k+\alpha Nz_k
  =z_k+2\alpha(y_k-x_k).
\]
The choice $\alpha=1/2$ gives the usual Douglas--Rachford update, whereas $\alpha=1$ gives the Peaceman--Rachford map $N$.  In the general convex setting below, $\alpha\in(0,1)$ makes $T$ an averaged map.  Appendix~\ref{app:linear-drs} gives a larger admissible interval under strong convexity and smoothness.

\begin{proposition}[fixed points, primal--dual pairs, and minimizers]
\label{prop:fixed-point-solution}
Let $\alpha>0$.  Then $\Fix(T)=\Fix(N)$ and
\begin{equation}
\label{eq:DR-fixed-point-characterization}
  \Fix(N)
  =\left\{x+\gamma a:\
  a\in\subd f(x),\ -a\in\subd g(x)\right\}.
\end{equation}
Consequently,
\begin{equation}
\label{eq:DR-shadow-zero-set}
  P_{\gamma f}(\Fix(T))
  =\operatorname{zer}(\subd f+\subd g)
  \subseteq\argmin(f+g).
\end{equation}
If the subdifferential sum rule holds at every minimizer---for example, if
\begin{equation}
\label{eq:DR-sum-rule-qualification}
  0\in\operatorname{sri}(\dom f-\dom g),
\end{equation}
then the inclusion in \eqref{eq:DR-shadow-zero-set} is an equality.  Thus every minimizer is the proximal shadow of at least one Douglas--Rachford fixed point.
\end{proposition}

\begin{proof}
Because $Tz=z$ is equivalent to $\alpha(Nz-z)=0$ and $\alpha>0$, we first have $\Fix(T)=\Fix(N)$.

Let $z_\star\in\Fix(N)$ and define
\[
  x_\star=P_{\gamma f}z_\star,
  \qquad
  w_\star=R_{\gamma f}z_\star=2x_\star-z_\star,
  \qquad
  y_\star=P_{\gamma g}w_\star.
\]
The fixed-point identity $R_{\gamma g}w_\star=z_\star$ reads
$2y_\star-w_\star=z_\star$.  Substituting
$w_\star=2x_\star-z_\star$ gives $y_\star=x_\star$.
By the proximal optimality condition \eqref{eq:proximal-optimality},
\begin{equation}
\label{eq:DR-fixed-point-subgradients}
  a_\star:=\frac{z_\star-x_\star}{\gamma}
  \in\subd f(x_\star),
  \qquad
  \frac{w_\star-y_\star}{\gamma}
  =\frac{x_\star-z_\star}{\gamma}
  =-a_\star
  \in\subd g(x_\star).
\end{equation}
Hence $z_\star=x_\star+\gamma a_\star$ has the form displayed in \eqref{eq:DR-fixed-point-characterization}, and
$0\in\subd f(x_\star)+\subd g(x_\star)$.

Conversely, suppose that $a\in\subd f(x)$ and $-a\in\subd g(x)$, and define
\[
  z:=x+\gamma a,
  \qquad
  w:=x-\gamma a.
\]
The first inclusion and \eqref{eq:proximal-optimality} give $P_{\gamma f}z=x$, and therefore $R_{\gamma f}z=2x-z=w$.  The second inclusion gives $P_{\gamma g}w=x$, and hence
\[
  R_{\gamma g}R_{\gamma f}z
  =R_{\gamma g}w
  =2x-w
  =x+\gamma a
  =z.
\]
Thus $z\in\Fix(N)$ and $P_{\gamma f}z=x$.  This proves both \eqref{eq:DR-fixed-point-characterization} and the first equality in \eqref{eq:DR-shadow-zero-set}.

To verify the inclusion into the minimizer set directly, let
$a\in\subd f(x)$ and $-a\in\subd g(x)$.  For every $y\in\Hcal$, the two subgradient inequalities give
\[
  f(y)\ge f(x)+\ip a{y-x},
  \qquad
  g(y)\ge g(x)-\ip a{y-x}.
\]
Adding them yields $(f+g)(y)\ge(f+g)(x)$, so $x$ minimizes $f+g$.
Under \eqref{eq:DR-sum-rule-qualification}, the convex subdifferential sum rule gives
$\subd(f+g)(x)=\subd f(x)+\subd g(x)$.  Fermat's rule then implies that every minimizer satisfies
$0\in\subd f(x)+\subd g(x)$, which proves the reverse inclusion.  The qualification and this conclusion are among the standard sufficient conditions collected in \cite[Cor.~27.6]{BauschkeCombettes2017}.
\end{proof}

The next elementary closedness property is used in the convergence proof.

\begin{lemma}[demiclosedness of the fixed-point residual]
\label{lem:demiclosedness}
Let $S:\Hcal\to\Hcal$ be nonexpansive.  If $u_n\rightharpoonup u$ and
$u_n-Su_n\to0$, then $Su=u$.
\end{lemma}

\begin{proof}
Set $r_n:=u_n-Su_n$.  Weak convergence makes $(u_n)$ bounded, and nonexpansiveness makes $(Su_n)$ bounded.  Since $r_n\to0$,
\[
  \norm{u_n-Su}^2
  =\norm{Su_n-Su+r_n}^2
  =\norm{Su_n-Su}^2+o(1)
  \le\norm{u_n-u}^2+o(1).
\]
On the other hand,
\[
  \norm{u_n-Su}^2-\norm{u_n-u}^2
  =2\ip{u_n-u}{u-Su}+\norm{u-Su}^2
  \longrightarrow\norm{u-Su}^2.
\]
The preceding upper bound therefore implies $\norm{u-Su}^2\le0$, so $Su=u$.  This is the demiclosedness principle for nonexpansive maps; see also \cite[Cor.~4.28]{BauschkeCombettes2017}.
\end{proof}

We will also use the following standard coupled graph-limit principle for maximal monotone operators.  Stating it explicitly records the precise technical fact needed to pass from the fixed-point sequence to the proximal shadow sequence.

\begin{fact}[coupled graph-limit principle]
\label{fact:coupled-graph-limit}
Let $A,B:\Hcal\rightrightarrows\Hcal$ be maximally monotone.  Suppose
$(x_n,a_n)\in\operatorname{gra}A$ and
$(y_n,b_n)\in\operatorname{gra}B$, and suppose that
\[
  x_n\rightharpoonup x,
  \qquad
  a_n\rightharpoonup a,
  \qquad
  x_n-y_n\to0,
  \qquad
  a_n+b_n\to0.
\]
Then $(x,a)\in\operatorname{gra}A$ and
$(x,-a)\in\operatorname{gra}B$; in particular,
$0\in Ax+Bx$.  This result is \cite[Cor.~26.6]{BauschkeCombettes2017} specialized to the identity coupling.  The strong consensus and balance limits are what make the simultaneous weak graph passage valid.
\end{fact}

\begin{theorem}[exact convergence in the general convex case]
\label{thm:general-convex-DR}
Assume $f,g\in\Gamma_0(\Hcal)$,
$\Fix(N)\ne\emptyset$, and $\alpha\in(0,1)$.  Starting from any
$z_0\in\Hcal$, generate $(x_k,y_k,z_{k+1})$ by
\eqref{eq:DR-proxf-first-x}--\eqref{eq:DR-proxf-first-z}.  Then:
\begin{enumerate}[label=\textup{(\roman*)}]
  \item for every $\bar z\in\Fix(N)$,
  \begin{equation}
  \label{eq:DR-Fejer-estimate}
    \norm{z_{k+1}-\bar z}^2
    \le\norm{z_k-\bar z}^2
    -\alpha(1-\alpha)\norm{Nz_k-z_k}^2;
  \end{equation}
  \item $Nz_k-z_k=2(y_k-x_k)\to0$;
  \item there exists $z_\star\in\Fix(N)$ such that
  $z_k\rightharpoonup z_\star$;
  \item with $x_\star=P_{\gamma f}z_\star$,
  \[
    x_k\rightharpoonup x_\star,
    \qquad
    y_k\rightharpoonup x_\star,
    \qquad
    0\in\subd f(x_\star)+\subd g(x_\star),
  \]
  and therefore $x_\star$ minimizes $f+g$.
\end{enumerate}
If $\Hcal$ is finite-dimensional, all three weak convergences above are strong.
\end{theorem}

\begin{proof}
By \cref{lem:reflected-nonexpansive}, both $R_{\gamma f}$ and
$R_{\gamma g}$ are nonexpansive.  Their composition $N$ is therefore nonexpansive:
\[
  \norm{Nu-Nv}
  \le\norm{R_{\gamma f}u-R_{\gamma f}v}
  \le\norm{u-v}.
\]
Fix $\bar z\in\Fix(N)$.  The Hilbert-space identity
\[
  \norm{(1-\alpha)a+\alpha b}^2
  =(1-\alpha)\norm a^2+\alpha\norm b^2
   -\alpha(1-\alpha)\norm{a-b}^2
\]
applied to $a=z_k-\bar z$ and $b=Nz_k-N\bar z$ gives
\begin{align*}
  \norm{z_{k+1}-\bar z}^2
  &=(1-\alpha)\norm{z_k-\bar z}^2
    +\alpha\norm{Nz_k-N\bar z}^2\\
  &\quad
    -\alpha(1-\alpha)\norm{Nz_k-z_k}^2\\
  &\le\norm{z_k-\bar z}^2
    -\alpha(1-\alpha)\norm{Nz_k-z_k}^2.
\end{align*}
This proves \eqref{eq:DR-Fejer-estimate}.  Hence $(z_k)$ is Fej\'er monotone with respect to $\Fix(N)$ and is bounded.  Summing \eqref{eq:DR-Fejer-estimate} over $k$ gives
\[
  \alpha(1-\alpha)
  \sum_{k=0}^{\infty}\norm{Nz_k-z_k}^2
  \le\norm{z_0-\bar z}^2,
\]
so $Nz_k-z_k\to0$.  The explicit DRS formulas give
$Nz_k-z_k=2(y_k-x_k)$, proving (ii).

Every bounded sequence in a Hilbert space has a weakly convergent subsequence.  Let $z_{k_j}\rightharpoonup z$.  Since
$z_{k_j}-Nz_{k_j}\to0$, \cref{lem:demiclosedness} gives $z\in\Fix(N)$.  Thus every weak sequential cluster point of $(z_k)$ is a fixed point.
For each $c\in\Fix(N)$, the Fej\'er estimate shows that
$\norm{z_k-c}$ has a limit.  If $z$ and $\widehat z$ were two weak cluster points, then the convergent sequence
\[
  \norm{z_k-z}^2-\norm{z_k-\widehat z}^2
  =2\ip{z_k}{\widehat z-z}+\norm z^2-\norm{\widehat z}^2
\]
would have limit $-\norm{z-\widehat z}^2$ along a subsequence converging weakly to $z$ and limit $+\norm{z-\widehat z}^2$ along one converging weakly to $\widehat z$.  Hence $z=\widehat z$.  The bounded sequence $(z_k)$ therefore has a unique weak cluster point, denoted by $z_\star$, and $z_k\rightharpoonup z_\star\in\Fix(N)$.  This is the Krasnosel'skii--Mann argument in explicit form; compare \cite[Thm.~5.15]{BauschkeCombettes2017}.

It remains to identify the shadows.  Define the exact subgradients furnished by the two proximal calls:
\begin{equation}
\label{eq:DR-shadow-graph-pairs}
  a_k:=\frac{z_k-x_k}{\gamma}\in\subd f(x_k),
  \qquad
  b_k:=\frac{2x_k-z_k-y_k}{\gamma}\in\subd g(y_k).
\end{equation}
They satisfy
\begin{equation}
\label{eq:DR-shadow-balance}
  a_k+b_k=\frac{x_k-y_k}{\gamma}\longrightarrow0.
\end{equation}
The sequence $(x_k)$ is bounded because $P_{\gamma f}$ is nonexpansive and $(z_k)$ is bounded.  Consequently $(a_k)$ is bounded, and then so are $(y_k)$ and $(b_k)$ by
$y_k-x_k\to0$ and \eqref{eq:DR-shadow-balance}.

Choose a weakly convergent subsequence of the bounded pairs $(x_k,a_k)$, say
$x_{k_j}\rightharpoonup\bar x$ and
$a_{k_j}\rightharpoonup\bar a$.  Then
$y_{k_j}\rightharpoonup\bar x$ and
$b_{k_j}\rightharpoonup-\bar a$.  Applying
\cref{fact:coupled-graph-limit} with
$A=\subd f$ and $B=\subd g$ gives
\[
  \bar a\in\subd f(\bar x),
  \qquad
  -\bar a\in\subd g(\bar x).
\]
Moreover, $z_k=x_k+\gamma a_k$ and $z_k\rightharpoonup z_\star$, so
$z_\star=\bar x+\gamma\bar a$.  By
\eqref{eq:proximal-optimality}, this implies
$\bar x=P_{\gamma f}z_\star=:x_\star$.
Thus every weak cluster point of $(x_k)$ is the same point $x_\star$, and the bounded sequence $x_k$ converges weakly to $x_\star$.  Since $y_k-x_k\to0$, also $y_k\rightharpoonup x_\star$.  The subgradient inclusions show that
$0\in\subd f(x_\star)+\subd g(x_\star)$, and
\cref{prop:fixed-point-solution} shows that $x_\star$ minimizes $f+g$.
This shadow conclusion is the content of the classical Douglas--Rachford theorem; see \cite[Thm.~26.11 and Cor.~28.3]{BauschkeCombettes2017} and \cite{EcksteinBertsekas1992}.

Finally, in finite-dimensional Hilbert spaces weak and strong convergence coincide, which gives the last assertion.
\end{proof}

\section{Linear convergence for Douglas--Rachford}
\label{app:linear-drs}

This appendix derives the primal contraction estimate underlying the linear Douglas--Rachford bounds of Giselsson and Boyd \cite{GiselssonBoyd2017}.  The proof is included because the inexact results in the body of the paper perturb this exact contraction.

\begin{assumption}[strong convexity and smoothness of the smooth term]
\label{ass:GB-primal}
The functions $f,g\in\Gamma_0(\Hcal)$.  In addition, $f$ is Fr\'echet differentiable, $\sigma$-strongly convex, and $\beta$-smooth for constants
\[
  0<\sigma\le\beta<\infty.
\]
Equivalently, its gradient is $\sigma$-strongly monotone and
$\beta$-Lipschitz:
\begin{align}
\label{eq:strong-gradient-monotonicity}
  \ip{\grad f(u)-\grad f(v)}{u-v}
  &\ge\sigma\norm{u-v}^2,\\
\label{eq:gradient-smoothness}
  \norm{\grad f(u)-\grad f(v)}
  &\le\beta\norm{u-v}
\end{align}
for every $u,v\in\Hcal$.
\end{assumption}

The additional information in \eqref{eq:strong-gradient-monotonicity}--\eqref{eq:gradient-smoothness} sharpens monotonicity into a two-sided sector constraint.  The main tool for combining the lower and upper curvature bounds is the following form of the Baillon--Haddad theorem.

\begin{fact}[Baillon--Haddad cocoercivity]
\label{fact:Baillon-Haddad}
Let $\psi:\Hcal\to\R$ be convex and Fr\'echet differentiable, and suppose that $\grad\psi$ is $L$-Lipschitz for some $L>0$.  Then
\begin{equation}
\label{eq:Baillon-Haddad}
  \ip{\grad\psi(u)-\grad\psi(v)}{u-v}
  \ge\frac{1}{L}
  \norm{\grad\psi(u)-\grad\psi(v)}^2
  \qquad(u,v\in\Hcal).
\end{equation}
Thus a smooth convex gradient is not merely monotone: its inner product with the displacement controls the squared norm of the gradient change.  In operator terminology, $\grad\psi$ is $1/L$-cocoercive.  This is \cite[Cor.~18.17]{BauschkeCombettes2017}.
\end{fact}

\begin{lemma}[sector inequality for strongly convex smooth gradients]
\label{lem:sector}
Let $f:\Hcal\to\R$ be $a$-strongly convex and $b$-smooth, with
$0<a\le b$.  For $u,v\in\Hcal$, set
$s=u-v$ and $y=\grad f(u)-\grad f(v)$.  Then
\begin{equation}
\label{eq:sector-ineq}
  \ip{s}{y}
  \ge
  \frac{ab}{a+b}\norm{s}^2
  +\frac{1}{a+b}\norm{y}^2.
\end{equation}
Equivalently,
\begin{equation}
\label{eq:sector-ball-form}
  \norm{y-\tfrac{a+b}{2}s}
  \le\frac{b-a}{2}\norm{s}.
\end{equation}
The gradient difference therefore lies in the closed sector bounded by the slopes $a$ and $b$.
\end{lemma}

\begin{proof}
First suppose $a<b$ and introduce the shifted function
\begin{equation}
\label{eq:shifted-function}
  \psi(x):=f(x)-\frac{a}{2}\norm{x}^2.
\end{equation}
Strong convexity of $f$ means exactly that $\psi$ is convex.  Smoothness of $f$ implies that $b\norm{\cdot}^2/2-f$ is convex; hence
\[
  \frac{b-a}{2}\norm{\cdot}^2-\psi
  =\frac{b}{2}\norm{\cdot}^2-f
\]
is convex.  The smooth-convex equivalences in
\cite[Thm.~18.15]{BauschkeCombettes2017} therefore show that
$\psi$ is $(b-a)$-smooth.  This explains why the shift removes the lower curvature $a$ and leaves the residual curvature interval $[0,b-a]$.

Now
\[
  \grad\psi(u)-\grad\psi(v)=y-a s.
\]
Applying the Baillon--Haddad estimate \eqref{eq:Baillon-Haddad} with
$L=b-a$ gives
\begin{equation}
\label{eq:BH-shifted}
  \ip{s}{y-a s}
  \ge\frac{1}{b-a}\norm{y-a s}^2.
\end{equation}
Multiplying by $b-a$ and expanding both sides yields
\begin{align*}
  (b-a)\bigl(\ip{s}{y}-a\norm{s}^2\bigr)
  &\ge\norm{y-a s}^2\\
  &=\norm y^2-2a\ip{s}{y}+a^2\norm s^2.
\end{align*}
Collecting the terms containing $\ip{s}{y}$ gives
\[
  (a+b)\ip{s}{y}
  \ge ab\norm s^2+\norm y^2,
\]
which is \eqref{eq:sector-ineq}.

If $a=b$, the function $\psi$ in \eqref{eq:shifted-function} is $0$-smooth, so its gradient is constant.  Hence
$y-a s=0$, and \eqref{eq:sector-ineq} holds with equality.
Finally, expanding \eqref{eq:sector-ball-form}, squaring both sides, and simplifying gives exactly
$(a+b)\ip{s}{y}\ge ab\norm s^2+\norm y^2$, so the two forms are equivalent.
\end{proof}

\begin{theorem}[contraction of the reflected smooth proximal map]
\label{thm:reflected-prox-contraction}
Under Assumption~\ref{ass:GB-primal}, for every $\gamma>0$ the reflected proximal map
$R_{\gamma f}=2P_{\gamma f}-\Id$ is contractive with factor
\begin{equation}
\label{eq:delta-primal}
  \delta
  :=\max\left\{
  \frac{\gamma\beta-1}{\gamma\beta+1},
  \frac{1-\gamma\sigma}{1+\gamma\sigma}
  \right\}
  =\max_{\lambda\in[\sigma,\beta]}
  \left|\frac{1-\gamma\lambda}{1+\gamma\lambda}\right|.
\end{equation}
In particular, $0\le\delta<1$ and
\begin{equation}
\label{eq:Rf-contractive}
  \norm{R_{\gamma f}u-R_{\gamma f}v}
  \le\delta\norm{u-v}
  \qquad(u,v\in\Hcal).
\end{equation}
The factor \eqref{eq:delta-primal} is the sharp uniform bound for this strongly convex, smooth class; see \cite{GiselssonBoyd2017}.
\end{theorem}

\begin{proof}
Let
$p=P_{\gamma f}u$ and $q=P_{\gamma f}v$.  Since $f$ is differentiable, the proximal optimality condition \eqref{eq:proximal-optimality} becomes
\[
  u-p=\gamma\grad f(p),
  \qquad
  v-q=\gamma\grad f(q).
\]
Set
\[
  s:=p-q,
  \qquad
  y:=\grad f(p)-\grad f(q).
\]
Subtracting the two optimality equations and using the definition of the reflected proximal map gives
\begin{equation}
\label{eq:u-v-s-y}
  u-v=s+\gamma y,
  \qquad
  R_{\gamma f}u-R_{\gamma f}v=s-\gamma y.
\end{equation}
Thus the desired contraction is exactly the comparison
\begin{equation}
\label{eq:ratio-target}
  \norm{s-\gamma y}
  \le\delta\norm{s+\gamma y}.
\end{equation}
The vectors on the two sides are the reflected and original displacements, respectively.

The scalar function
$\lambda\mapsto |1-\gamma\lambda|/(1+\gamma\lambda)$ decreases on
$(0,1/\gamma]$ and increases on $[1/\gamma,\infty)$.  Its maximum on the interval $[\sigma,\beta]$ is therefore attained at an endpoint, which explains the second expression in \eqref{eq:delta-primal}.  At least one of the two displayed endpoint terms is nonnegative, and each absolute endpoint ratio is strictly below one; hence $0\le\delta<1$.

To convert this scalar endpoint bound into a bound for the nonlinear gradient sector, define
\begin{equation}
\label{eq:a-b-delta}
  a_\delta:=\frac{1-\delta}{\gamma(1+\delta)},
  \qquad
  b_\delta:=\frac{1+\delta}{\gamma(1-\delta)}.
\end{equation}
These are the lower and upper scalar curvatures for which the reflected scalar ratio has magnitude exactly $\delta$.  The definition of $\delta$ implies
\begin{equation}
\label{eq:interval-contained}
  a_\delta\le\sigma\le\beta\le b_\delta.
\end{equation}
Indeed,
$\delta\ge(1-\gamma\sigma)/(1+\gamma\sigma)$ is equivalent to
$a_\delta\le\sigma$, while
$\delta\ge(\gamma\beta-1)/(\gamma\beta+1)$ is equivalent to
$\beta\le b_\delta$.
Therefore $f$ is also $a_\delta$-strongly convex and
$b_\delta$-smooth.  Applying \cref{lem:sector} to the pair $(s,y)$ with these enlarged sector endpoints gives
\begin{equation}
\label{eq:sector-delta}
  \ip{s}{y}
  \ge
  \frac{a_\delta b_\delta}{a_\delta+b_\delta}\norm{s}^2
  +\frac{1}{a_\delta+b_\delta}\norm{y}^2.
\end{equation}
A direct calculation from \eqref{eq:a-b-delta} gives
\begin{equation}
\label{eq:delta-ab-calc}
  a_\delta b_\delta=\frac{1}{\gamma^2},
  \qquad
  \frac{1}{a_\delta+b_\delta}
  =\frac{\gamma(1-\delta^2)}{2(1+\delta^2)}.
\end{equation}
Substitution into \eqref{eq:sector-delta} yields
\begin{equation}
\label{eq:key-c-delta}
  \ip{s}{y}
  \ge
  \frac{1-\delta^2}{2\gamma(1+\delta^2)}
  \left(\norm{s}^2+\gamma^2\norm{y}^2\right).
\end{equation}
Finally, expand the difference of the two squared sides of \eqref{eq:ratio-target}:
\begin{align*}
  \norm{s-\gamma y}^2
  -\delta^2\norm{s+\gamma y}^2
  &=(1-\delta^2)
    \left(\norm{s}^2+\gamma^2\norm{y}^2\right)\\
  &\quad
    -2\gamma(1+\delta^2)\ip{s}{y}\\
  &\le0
\end{align*}
by \eqref{eq:key-c-delta}.  This proves \eqref{eq:ratio-target} and hence \eqref{eq:Rf-contractive}.
\end{proof}

\begin{theorem}[exact linear convergence of primal Douglas--Rachford]
\label{thm:exact-linear-DR}
Assume that Assumption~\ref{ass:GB-primal} holds.  Fix $\gamma>0$, let $\delta$ be given by \eqref{eq:delta-primal}, and choose
\begin{equation}
\label{eq:alpha-range-primal}
  \alpha\in\left(0,\frac{2}{1+\delta}\right).
\end{equation}
Then the exact Douglas--Rachford operator
$T=(1-\alpha)\Id+\alpha R_{\gamma g}R_{\gamma f}$ is a contraction with factor
\begin{equation}
\label{eq:q-primal}
  q:=|1-\alpha|+\alpha\delta<1.
\end{equation}
It has a unique fixed point $z_\star$, and the iterates satisfy
\begin{equation}
\label{eq:exact-z-linear}
  \norm{z_{k+1}-z_\star}
  \le q\norm{z_k-z_\star}
  \le q^{k+1}\norm{z_0-z_\star}.
\end{equation}
The proximal shadow $x_k=P_{\gamma f}z_k$ converges $R$-linearly to the unique minimizer
$x_\star=P_{\gamma f}z_\star$, with the sharper estimate
\begin{equation}
\label{eq:exact-x-linear}
  \norm{x_k-x_\star}
  \le\frac{1}{1+\gamma\sigma}\norm{z_k-z_\star}
  \le\frac{q^k}{1+\gamma\sigma}\norm{z_0-z_\star}.
\end{equation}
\end{theorem}

\begin{proof}
By \cref{thm:reflected-prox-contraction}, $R_{\gamma f}$ is
$\delta$-contractive.  By \cref{lem:reflected-nonexpansive},
$R_{\gamma g}$ is nonexpansive.  Consequently their composition
$C:=R_{\gamma g}R_{\gamma f}$ is $\delta$-contractive:
\[
  \norm{Cu-Cv}
  \le\norm{R_{\gamma f}u-R_{\gamma f}v}
  \le\delta\norm{u-v}.
\]
For arbitrary $u,v\in\Hcal$, the triangle inequality gives
\begin{align*}
  \norm{Tu-Tv}
  &=\norm{(1-\alpha)(u-v)+\alpha(Cu-Cv)}\\
  &\le |1-\alpha|\norm{u-v}
       +\alpha\norm{Cu-Cv}\\
  &\le\bigl(|1-\alpha|+\alpha\delta\bigr)\norm{u-v}.
\end{align*}
If $0<\alpha\le1$, then
$q=1-\alpha+\alpha\delta=1-\alpha(1-\delta)<1$.
If $\alpha>1$, then
$q=\alpha(1+\delta)-1<1$ exactly when
$\alpha<2/(1+\delta)$.  Thus \eqref{eq:alpha-range-primal} is precisely the range in which $T$ is a strict contraction.

Because a Hilbert space is complete, the Banach--Picard fixed-point theorem gives a unique $z_\star\in\Fix(T)$ and the one-step and iterated estimates in \eqref{eq:exact-z-linear}; see also \cite[Thm.~1.50]{BauschkeCombettes2017}.  By \cref{prop:fixed-point-solution},
$x_\star=P_{\gamma f}z_\star$ satisfies
$0\in\subd f(x_\star)+\subd g(x_\star)$ and therefore minimizes $f+g$.  Since $f$ is strongly convex and $g$ is convex, this minimizer is unique.

It remains to establish the sharpened shadow estimate.  Let
$p=P_{\gamma f}u$ and $q=P_{\gamma f}v$.  From the two proximal optimality conditions and the $\sigma$-strong monotonicity of $\grad f$,
\begin{align*}
  \ip{u-v}{p-q}
  &=\norm{p-q}^2
    +\gamma\ip{\grad f(p)-\grad f(q)}{p-q}\\
  &\ge(1+\gamma\sigma)\norm{p-q}^2.
\end{align*}
Cauchy--Schwarz therefore yields
\begin{equation}
\label{eq:strong-prox-Lipschitz}
  \norm{P_{\gamma f}u-P_{\gamma f}v}
  \le\frac{1}{1+\gamma\sigma}\norm{u-v}.
\end{equation}
Applying \eqref{eq:strong-prox-Lipschitz} to $u=z_k$ and $v=z_\star$, and then using \eqref{eq:exact-z-linear}, proves \eqref{eq:exact-x-linear}.
\end{proof}

\begin{corollary}[optimal Giselsson--Boyd parameters in the primal bound]
\label{cor:optimal-primal}
Under Assumption~\ref{ass:GB-primal}, the contraction bound in
\cref{thm:exact-linear-DR} is minimized by
\begin{equation}
\label{eq:gamma-alpha-opt-primal}
  \alpha_\star=1,
  \qquad
  \gamma_\star=\frac{1}{\sqrt{\sigma\beta}}.
\end{equation}
The optimized factor is
\begin{equation}
\label{eq:optimal-factor-primal}
  q_\star
  =\frac{\sqrt{\kappa}-1}{\sqrt{\kappa}+1},
  \qquad
  \kappa:=\frac{\beta}{\sigma}.
\end{equation}
\end{corollary}

\begin{proof}
For fixed $\gamma$, the factor
$q(\alpha)=|1-\alpha|+\alpha\delta$ is affine with slope
$-(1-\delta)<0$ on $(0,1]$ and affine with slope
$1+\delta>0$ on $[1,2/(1+\delta))$.  It is therefore minimized at
$\alpha=1$, where $q=\delta$.

It remains to minimize $\delta$ over $\gamma>0$.  The endpoint function
\[
  \gamma\longmapsto\frac{\gamma\beta-1}{\gamma\beta+1}
\]
is strictly increasing, whereas
\[
  \gamma\longmapsto\frac{1-\gamma\sigma}{1+\gamma\sigma}
\]
is strictly decreasing.  The maximum of the two is minimized where they are equal:
\[
  \frac{\gamma\beta-1}{\gamma\beta+1}
  =\frac{1-\gamma\sigma}{1+\gamma\sigma}.
\]
Cross-multiplication gives $\gamma^2\sigma\beta=1$, hence
$\gamma_\star=1/\sqrt{\sigma\beta}$.  Substituting this value into either endpoint term gives
\[
  \delta_\star
  =\frac{\sqrt{\beta/\sigma}-1}
         {\sqrt{\beta/\sigma}+1}
  =\frac{\sqrt\kappa-1}{\sqrt\kappa+1},
\]
which is \eqref{eq:optimal-factor-primal}.  Giselsson and Boyd show that this class bound is tight \cite{GiselssonBoyd2017}.
\end{proof}

\bibliographystyle{plainnat}
\bibliography{cr_drs}

\end{document}